\documentclass{article}
\usepackage{amsmath,amsthm,amsfonts,amssymb,enumitem,mathtools}
\usepackage{url}
\usepackage{fullpage}
\usepackage{hyperref}
\usepackage{graphicx}
\usepackage{array}
\usepackage{float}
\usepackage{setspace}

\usepackage{multirow}
\usepackage{longtable}

\newtheorem{theorem}{Theorem}[section]
\newtheorem{dfn}[theorem]{Definition}

\newtheorem{cor}[theorem]{Corollary}
\newtheorem{lem}[theorem]{Lemma}
\newtheorem{prop}[theorem]{Proposition}
\newtheorem*{remark}{Remark}

\theoremstyle{definition}

\DeclareMathOperator\tr{tr}

\begin{document}

\title{Equiangular lines in low dimensional Euclidean spaces}
\author{Gary R. W. Greaves \thanks{ Division of Mathematical Sciences,
  School of Physical and Mathematical Sciences,
  Nanyang Technological University,
  21 Nanyang Link, Singapore 637371, Singapore
  {\tt gary@ntu.edu.sg}. Supported by the Singapore
Ministry of Education Academic Research Fund (Tier 1);  grant number:  RG29/18.}
\and
Jeven Syatriadi \thanks{ Division of Mathematical Sciences, 
  School of Physical and Mathematical Sciences, 
  Nanyang Technological University, 
  21 Nanyang Link, Singapore 637371, Singapore.
  {\tt jsyatriadi@ntu.edu.sg}.}
\and
  Pavlo Yatsyna \thanks{Department of Algebra, 
Faculty of Mathematics and Physics, 
Charles University, 
Sokolovsk\'a 83, 18600 Praha 8, Czech Republic.
  {\tt pvyatsyna@gmail.com}. Supported by project PRIMUS/20/SCI/002 from Charles University.}
}
\date{}
\maketitle

\begin{abstract}
    We show that the maximum cardinality of an equiangular line system in $14$ and $16$ dimensions is $28$ and $40$, respectively, thereby solving a longstanding open problem.
    We also improve the upper bounds on the cardinality of equiangular line systems in $19$ and $20$ dimensions to $74$ and $94$, respectively.
    
\end{abstract}

\section{Introduction}

A set of lines each passing through the origin of Euclidean space is called an \textbf{equiangular line system} if the angle between any pair of lines is the same.
Given $d$, we would like to know $N(d)$, the maximum cardinality of an equiangular line system in $\mathbb R^d$.
This problem dates back to the late 1940s with a paper of Haantjes~\cite{haantjes48}, who determined $N(2)=3$ and $N(3) = 6$.
The study of equiangular line systems largely developed in the 1970s due to the advent of the linear algebraic approach of Seidel et al.~\cite{lemmens73,vLintSeidel66,Seidel74}. 
In particular, in 1973, Lemmens and Seidel~\cite{lemmens73} determined the sequence $(N(d))_{d \in \mathbb N}$ consecutively for $d$ up to $13$.
(See Sequence A002853 in The On-Line Encyclopedia of Integer Sequences.)

Up until now, and despite a considerable amount of research in the
past 45 years, determining the value of $N(14)$ had remained an open problem.
In this paper we show that $N(14) = 28$ and $N(16) = 40$ and we improve the upper bounds for dimensions $d=19$ and $d=20$, i.e., we show that $N(19) \leqslant 74$ and $N(20) \leqslant 94$.
Our work verifies part of a conjecture of Lin and Yu~\cite[Conjecture 3]{LinYu}.
Furthermore, we vindicate Seidel in \cite[Section 3.3]{SeidelIncorrectTable} where $N(14) = 28$ and $N(16) = 40$ were claimed without proof.
Note that Seidel also claimed that $N(18) = 48$, which we now know to be false~\cite{LinYu,Szo} (see Table~\ref{tab:equi}).

Equiangular line systems correspond to a variety of objects in different mathematical disciplines: regular two-graphs in group theory~\cite{taylor}, equilateral point-sets in elliptic geometry~\cite{haantjes48}, and optimal Grassmannian frames in frame theory~\cite{heath}.
Just in the last few years there have been many developments relating to the sequence $(N(d))_{d \in \mathbb N}$.
In particular, there have emerged improvements to the lower bounds for $N(18)$~\cite{LinYu,Szo} and improvements to the upper bounds for $N(d)$ where $d = 14, 16, 17, 18, 19, 20$ \cite{Azarija75,Azarija95,GG18,GKMS16,GreavesYatsyna19}.
There have also been various recent improvements to upper bounds for $N(d)$ for $d \geqslant 24$ using semidefinite programming, see \cite{GlazYu,KingTang,delaat,okudayu16, Yu17}. 

The asymptotic behaviour of $N(d)$ is quadratic in $d$ with a general upper bound of $d(d+1)/2$~\cite[Theorem 3.5]{lemmens73} and a general lower bound of $(32d^2 + 328d + 29)/1089$~\cite[Corollary 2.8]{GKMS16}.
One can also consider the related problem of, for fixed $\alpha \in (0,1)$, finding $N_\alpha(d)$, the maximum number of lines in $\mathbb R^d$ through the origin with pairwise angle $\arccos \alpha$. 
Motivated by a conjecture of Bukh~\cite{bukh16}, the asymptotic behaviour of $N_\alpha(d)$ was recently shown to be linear in $d$ \cite{balla18,jiang}.

In Table~\ref{tab:equi} below, we give the currently known (including the improvements from this paper) values or lower and upper bounds for $N(d)$ for $d$ at most $23$.
\begin{table}[ht]
	\begin{center}
	\setlength{\tabcolsep}{2pt}
	\begin{tabular}{c|ccccccccccccccc}
		$d$  & 2 & 3        & 4           & 5  & 6  & 7--14 & 15 & 16 & 17 & 18 & 19 & 20 & 21 & 22 & 23\\\hline
		$N(d)$  & 3 & 6        & 6           & 10 & 16 & 28  & 36 & 40 & 48--49 & 56--60 & 72--74 & 90--94  & 126 & 176 & 276
	\end{tabular}
	\end{center}
	\caption{Bounds for the sequence $N(d)$ for $2\leqslant d\leqslant 23$.  A single number is given in the cases where the exact number is known.  The improvements from this paper, $N(14) = 28$, $N(16) = 40$, $N(19) \leqslant 74$, and $N(20) \leqslant 94$ are included.}
	\label{tab:equi}
\end{table}

	Let $\mathcal L$ be an equiangular line system of cardinality $n$ in $\mathbb R^{d}$ with $n > d$.
	Suppose $\{ \mathbf{v}_1,\dots, \mathbf{v}_{n} \}$ is a set of unit spanning vectors for the lines in $\mathcal L$.
	For any two distinct vectors $\mathbf v_i$ and $\mathbf v_j$ (with $i \ne j$) the
	inner product $\mathbf{v}_i^\top \mathbf{v}_j$ is equal to $\pm \alpha$ for some $\alpha \in (0,1)$.
	Thus, the Gram matrix $G$ for this set of vectors has diagonal entries equal to $1$ and off-diagonal entries equal to $\pm \alpha$.
	The $\{0,\pm 1\}$-matrix $S = (G-I)/\alpha$ is called the \textbf{Seidel matrix} corresponding to the set of lines $\mathcal L$.
	Note that changing the direction of our spanning vectors $\mathbf v_i$ corresponds to conjugating $S$ with a $\{\pm 1\}$-diagonal matrix.
	Furthermore, since $G$ is positive semidefinite with rank $d$, the smallest eigenvalue of $S$ is $-1/\alpha$ with multiplicity $n-d$.
    
    To show the nonexistence of equiangular line systems of cardinality $n$ in $\mathbb R^d$ for certain pairs $(n,d)$, we demonstrate the nonexistence of their corresponding Seidel matrices.
    We take advantage of modular constraints on the coefficients of the characteristic polynomial of a Seidel matrix \cite{GreavesYatsyna19}.
    The eigenvalues of a Seidel matrix that corresponds to an equiangular line system of large cardinality relative to its ambient space are subject to strong geometric constraints.
    Our approach is to combine these modular and geometric constraints together to enable us to enumerate each possible characteristic polynomial for a putative Seidel matrix.
    Once an exhaustive list of the possible characteristic polynomials has been found, we apply spectral methods to show that no Seidel matrix can exist having the corresponding characteristic polynomials.
    
    To produce our exhaustive lists of possible characteristic polynomials, we use a polynomial enumeration algorithm, which we have implemented in SageMath~\cite{sage}.
    This implementation is available from GitHub~\cite{github}.
    The total running time (on a modern PC) of all the computations used in this paper is less than 22 minutes.
    
    The outline of this paper is as follows.
    In Section~\ref{sec:prelim}, we introduce (weakly)-type-2 polynomials and describe the polynomial enumeration algorithm.
    In Section~\ref{sec:ccp}, we enumerate all polynomials that potentially correspond to equiangular line systems for certain pairs $(n,d)$.
    In Section~\ref{sec:prescribed}, we present a method for demonstrating the nonexistence of Seidel matrices having a prescribed characteristic polynomial.
    In Section~\ref{sec:1920}, we show that $N(19) \leqslant 74$, and $N(20) \leqslant 94$ and in Section~\ref{sec:dim1416}, we show that $N(14) = 28$ and $N(16) = 40$.
\section{Preliminaries}

\label{sec:prelim}

\subsection{Characteristic polynomials and interlacing}
    
    Let $\mathcal L$ be an equiangular line system of cardinality $n$ in $\mathbb R^{d}$ with $n > d$ and let $S$ be the Seidel matrix corresponding to $\mathcal L$.
	The Seidel matrix $S$ is a symmetric matrix with real entries, which means that each zero of its characteristic polynomial $\chi_S(x) = \det(xI-S)$ is real.
	In other words, $\chi_S(x)$ is a \textbf{totally-real} polynomial.
	Moreover, since every entry of $S$ is an integer, each coefficient of $\chi_S(x)$ is also an integer.
	
	By \cite[Theorem 3.4]{lemmens73} together with \cite[Lemma 6.1]{vLintSeidel66} and \cite[Theorem 4.5]{lemmens73}, for each ordered pair $(n,d) \in \{(29,14),(41,16),(75,19),(95,20)\}$, any corresponding Seidel matrix must have smallest eigenvalue equal to $-5$.
	We record this result as a theorem.
	\begin{theorem}
	\label{thm:smallesteigenvalue}
	    Let $(n,d) \in \{(29,14),(41,16),(75,19),(95,20)\}$ and let $S$ be a Seidel matrix corresponding to $n$ equiangular lines in $\mathbb R^d$.
	    Then $\chi_S(x)=(x+5)^{n-d}p(x)$, where $p(x)$ is a monic integer polynomial of degree $d$ all of whose zeros are greater than $-5$.
	\end{theorem}

Our main approach for showing that a Seidel matrix $S$ having a certain spectrum does not exist is to consider the principal submatrices of $S$ and their characteristic polynomials. 
Cauchy's interlacing theorem, below, provides bounds for the eigenvalues of principal submatrices of $S$.
For a matrix $M$, denote by $M[i]$ the principal submatrix of $M$ obtained by deleting its $i$th row and column.

		\begin{theorem}[\cite{Cau:Interlace,Fisk:Interlace05, Hwang:Interlace04}]\label{thm:cauchyinterlace}
			Let $M$ be a real symmetric matrix having eigenvalues $\lambda_1 \leqslant \lambda_2 \leqslant \dots \leqslant \lambda_n$ and suppose $M[j]$, for some $j \in \{1,\dots,n \}$, has eigenvalues $\mu_1 \leqslant \mu_2 \leqslant \dots \leqslant \mu_{n-1}$.
			Then
			\[
				\lambda_1 \leqslant \mu_1 \leqslant \lambda_2 \leqslant \dots \leqslant \lambda_{n-1} \leqslant \mu_{n-1} \leqslant \lambda_n.
			\]
		\end{theorem}

	    Given $\delta \in \mathbb N$ and polynomials $f(x) = \prod_{i=1}^{\delta}(x-\lambda_i)$ and $g(x) = \prod_{i=0}^{\delta}(x-\mu_i)$ such that $\lambda_1 \leqslant \lambda_2 \leqslant \dots \leqslant \lambda_\delta$, and $\mu_0 \leqslant \mu_1 \leqslant \dots \leqslant \mu_\delta$, we say that $f$ \textbf{interlaces} $g$ if $\mu_0 \leqslant \lambda_1 \leqslant \mu_1 \leqslant \dots \leqslant \lambda_\delta \leqslant \mu_\delta$.

The next result is a condition on the sum of the characteristic polynomials of principal submatrices of a matrix.

\begin{theorem}[{\cite[Page 116]{thompson1968principal}}]\label{thm:sumofsubpoly}
Let $M$ be a real symmetric matrix of order $n$.
Then
\begin{equation}
    \label{eqn:thm}
    \sum_{i=1}^n \chi_{M[i]}(x)=\frac{\mathrm{d}}{\mathrm{d}x} \chi_M(x).
\end{equation}
\end{theorem}

Since every Seidel matrix $S$ of order $n$ has zero trace and the trace of $S^2$ is $n(n-1)$, we have the following.

		\begin{lem}[{\cite[Lemma 5.4]{GreavesYatsyna19}}]\label{lem:coeff_of_interlacepoly}
		Let $S$ be a Seidel matrix of order $n$. 
		Suppose $S$ has minimal polynomial $m_S(x) = \sum_{i=0}^d a_i x^{d-i}$.
		Then, for all $j \in \{1,\dots,n\}$,
		\[
		\chi_{S[j]}(x)=\dfrac{\chi_S(x)}{m_S(x)}\sum_{i=0}^{d-1} b_i x^{d-1-i},
		\]
		where $b_0 = 1$, $b_1 = a_1$, $b_2 = a_2+n-1$, and $b_i \in \mathbb Z$ for $i \in \{3,\dots,d-1\}$.
		\end{lem}

\subsection{Type-2 polynomials}

It was shown in \cite{GreavesYatsyna19} that the coefficients of characteristic polynomials of Seidel matrices satisfy certain modular constraints.

\begin{theorem}[{\cite[Section 3]{GreavesYatsyna19}}]\label{thm:basicDivisibility}
	Let $S$ be a Seidel matrix of order $n$ and write $\chi_{S+I}(x) = \sum_{i=0}^n a_i x^{n-i}$.
	Then $a_0 = 1$, $a_1 = -n$, and $a_2 = 0$.
	Furthermore, if $n$ is even then $2^{i}$ divides $a_i$ for all $i \in \{0,\dots,n\}$.
	Otherwise, $2^{i-1}$ divides $a_i$ for all odd $i \in \{0,\dots,n\}$ and $2^{i}$ divides $a_i$ for all even $i \in \{0,\dots,n\}$.
\end{theorem}

Motivated by the above theorem, we consider polynomials whose coefficients satisfy related modular conditions. 

\begin{dfn}\label{dfn:polytypes}
Let $p(x)=\sum_{i=0}^n a_i x^{n-i}$ be a monic polynomial in $\mathbb{Z}[x]$.
We say $p$ is \textbf{type $\mathbf 2$} if $2^i$ divides $a_i$ for all $i\geqslant 0$ and \textbf{weakly type $\mathbf 2$} if $2^{i-1}$ divides $a_i$ for all $i\geqslant 1$.
\end{dfn}

The next result follows from \cite[Lemma 3.1]{GreavesYatsyna19}.

\begin{lem}
\label{lem:shiftType}
Let $S$ be a Seidel matrix of order $n$ and $\kappa$ be an odd integer.
Then $\chi_{S-\kappa I}(x)$ is weakly type 2.
Furthermore, if $n$ is even then $\chi_{S-\kappa I}(x)$ is type 2.
\end{lem}

Note the following equivalent definition of (weakly)-type-2 polynomials.
A monic integer polynomial $p(x)$ is type 2 if and only if $p(2x)/2^{\deg p} \in \mathbb{Z}[x]$ and is weakly type 2 if and only if $p(2x)/2^{\deg p-1} \in \mathbb{Z}[x]$.

Recall that the \textbf{content} $c(p)$ of a polynomial $p \in \mathbb{Z}[x]$ is the greatest common divisor of its coefficients.
For $p\in \mathbb{Q}[x]$, the content $c(p)$ is defined to be $c(vp)/v$ where $v\in \mathbb{N}$ satisfying $vp \in \mathbb{Z}[x]$.
The following lemma deals with the factorisation of type-2 and weakly-type-2 polynomials.

\begin{lem}\label{lem:type_factor}
Let $p\in \mathbb{Z}[x]$ be a monic polynomial.
Suppose $p=qr$ where $q,r \in \mathbb{Z}[x]$.
Then 
\begin{itemize}
    \item $p$ is type 2 if and only if $q$ and $r$ are both type 2;
    \item $p$ is weakly type 2 if and only if $q$ and $r$ are both weakly type 2 and at least one of them is type 2.
\end{itemize}
\end{lem}
\begin{proof}
Since $p$ is monic, both $q$ and $r$ are also monic.
Observe that $p(2x)/2^{\deg p}=\left(q(2x)/2^{\deg q}\right) \cdot \left(r(2x)/2^{\deg r}\right)$ and both $q(2x)/2^{\deg q}$ and $r(2x)/2^{\deg r}$ are monic polynomials in $\mathbb{Q}[x]$.
It follows that there exist positive integers $u$ and $v$ such that $u \cdot q(2x)/2^{\deg q}$ and $v \cdot r(2x)/2^{\deg r}$ are both in $\mathbb{Z}[x]$ each with content equal to $1$.
Since the content is multiplicative, we obtain
\begin{align}
\frac{uv}{2} \cdot c\left(\frac{p(2x)}{2^{\deg p-1}}\right)=c\left(uv \cdot \frac{p(2x)}{2^{\deg p}}\right)=
c\left(u \frac{q(2x)}{2^{\deg q}}\right) \cdot c\left(v \frac{r(2x)}{2^{\deg r}}\right)=1.
\label{eq:gauss}
\end{align}

If $q$ and $r$ are both type 2 then $p(2x)/2^{\deg p}=\left(q(2x)/2^{\deg q}\right) \cdot \left(r(2x)/2^{\deg r}\right) \in \mathbb{Z}[x]$.
Hence $p$ is type 2.
Conversely, suppose $p$ is type 2.
Then $p(2x)/2^{\deg p}$ is a monic polynomial in $\mathbb{Z}[x]$, which implies $c\left(p(2x)/2^{\deg p}\right)=1$.
Consequently, $p(2x)/2^{\deg p-1} \in \mathbb{Z}[x]$ and $c\left(p(2x)/2^{\deg p-1}\right)=2$.
By \eqref{eq:gauss}, we have $uv=1$ and hence $u=v=1$.
Therefore $q$ and $r$ are both type 2.

If $q$ and $r$ are both weakly type 2 and at least one of them is type 2, then
$$\frac{p(2x)}{2^{\deg p-1}}=\frac{q(2x)}{2^{\deg q-1}} \cdot \frac{r(2x)}{2^{\deg r}}=\frac{q(2x)}{2^{\deg q}} \cdot \frac{r(2x)}{2^{\deg r-1}} \in \mathbb{Z}[x].$$
Hence, $p$ is weakly type 2.
Conversely, suppose $p$ is weakly type 2.
Then $p(2x)/2^{\deg p-1} \in \mathbb{Z}[x]$ with leading coefficient 2, which implies $c\left(p(2x)/2^{\deg p-1}\right)$ is equal to 1 or 2.
If $c\left(p(2x)/2^{\deg p-1}\right)=2$, then
$p(2x)/2^{\deg p} \in \mathbb{Z}[x]$, i.e., $p$ is type 2 and hence $q$ and $r$ are both type 2, as above.
Otherwise, we must have $c\left(p(2x)/2^{\deg p-1}\right)=1$.
Then $uv=2$ by \eqref{eq:gauss}.
Hence $\{u,v\} = \{1,2\}$, which implies that both $q$ and $r$ are weakly type 2 and one of them is type 2.
\end{proof}

Denote by $\mathcal C_n$ the set of all Seidel matrices of order $n$.
Given a positive integer $e$, define the set $\mathcal P_{n,e} = \{ \chi_S(x) \mod 2^e \mathbb Z[x] \; | \; S \in \mathcal C_n \}$.
We will require the following upper bound on the cardinality of $\mathcal P_{n,e}$ for odd $n$.

\begin{theorem}[{\cite[Corollary 3.13]{GreavesYatsyna19}}]\label{thm:countCharPolySeidelOdd}
			Let $n$ be an odd integer and $e$ be a positive integer.
			Then the cardinality of $\mathcal P_{n,e}$ is at most $2^{\binom{e-2}{2}+1}$.
\end{theorem}

In \cite{GreavesYatsyna19}, the authors conjectured\footnote{The necessary condition $e \geqslant 3$ was not stated in the conjecture given in \cite{GreavesYatsyna19}.
Indeed, for all natural numbers $n$ and $e$ with $e < 3$, it is easy to see that $|\mathcal P_{n,e}| < 2^{\binom{e-2}{2}+1}$.} that, for all integers $e \geqslant 3$, there exists $N \in  \mathbb N$ such that $|\mathcal P_{n,e}| = 2^{\binom{e-2}{2}+1}$ for all odd $n > N$.
The conjecture remains open.
However, for reasonably small values of $n$ and $e$ (i.e., $n \leqslant 95$ and $e \leqslant 7$) we can generate all elements of $\mathcal P_{n,e}$ by randomly generating Seidel matrices of order $n$ until we obtain $2^{\binom{e-2}{2}+1}$ characteristic polynomials in distinct congruence classes modulo $2^e \mathbb Z[x]$.

\subsection{Polynomial enumeration algorithm}
\label{sec:pea}

In the subsequent sections we frequently need to generate totally-real (weakly)-type-2 polynomials whose top three coefficients are fixed.
We use an algorithm developed by McKee and Smyth (see \cite[Section 3]{MS04} or \cite[Section 4.3]{GreavesYatsyna19}).
In each step of the above algorithm, a range of values is found for the constant term of a polynomial.
Our modification is to apply divisibility ``checks'' to reduce the number of possible values for the constant term at each iteration.
To illustrate how the algorithm works, we provide a toy example, below.

Suppose we want to find all polynomials $f(x)=x^4-18x^3+112x^2+a_3x+a_4$ such that all roots of $f$ are real and $f$ is type 2.
Since $f$ is totally real, the derivative $f'(x)=4x^3-54x^2+224x+a_3$ must also be totally real.
Hence, $a_3\in \{-294,\dots,-264\}$.
Now, we use the fact that $f$ is type 2, that is, $8$ divides $a_3$ and $16$ divides $a_4$.
Since 8 divides $a_3$, there are only four possibilities for $a_3$, which are $-288$, $-280$, $-272$, and $-264$.
For each $a_3$, we can find the range of possible values for $a_4$ that ensures that $f$ is totally real.
When $a_3 = -288$ we must have $a_4 \in \{ 256,\dots,262\}$.
When $a_3 = -280$ we must have $a_4 \in \{ 223,\dots,242\}$.
When $a_3 = -272$ we must have $a_4 \in \{ 185,\dots,194\}$.
And when $a_3 = -264$ we must have $a_4 \in \{ 144 \}$.
Now we impose the condition that 16 divides $a_4$.
In total, we obtain 5 possible polynomials for $f$:
\begin{align*}
x^4-18x^3+112x^2-288x+256,& \quad x^4-18x^3+112x^2-280x+224,\\
x^4-18x^3+112x^2-280x+240,& \quad x^4-18x^3+112x^2-272x+192,\\
x^4-18x^3+112x^2-264x+144.&
\end{align*}

We apply this algorithm repeatedly throughout.
The computations were executed on a modern PC in SageMath~\cite{sage}, with the output independently verified using  Mathematica~\cite{mathematica} and Pari/GP~\cite{parigp}.
The total running time for all computations used in this paper is less than 22 minutes.
The time taken to construct the sets $\mathcal P_{n,e}$ in the relevant cases is not included in the running times given below.
A SageMath implementation of this algorithm is available from GitHub~\cite{github}. 

\section{Candidate characteristic polynomials}
\label{sec:ccp}
In this section, we enumerate all candidates for the characteristic polynomial of a Seidel matrix corresponding to $n$ equiangular lines in $\mathbb R^d$, for $(n,d) \in \{(29,14),(41,16),(75,19),(95,20)\}$.

\begin{lem}\label{lem:ev_mult}
Let $S$ be a Seidel matrix of order $n$ with smallest eigenvalue $\lambda_0 \in \mathbb Z$ of multiplicity $n-d > 1$.
Let $\kappa$ be a closest odd integer to $(d-n)\lambda_0/d$.
Define
$$
\theta := \min \left \{ \eta \in \mathbb N \; | \; \eta4^{(\eta-\gamma(n))/\eta} > n(n-1)-\lambda_0^2(n-d)+2\kappa \lambda_0 (n-d)+d \kappa^2 \right \},
$$
where $\gamma(n) = 1$ if $n$ is odd and $\gamma(n) = 0$, otherwise.
If $\theta \leqslant d$ then
$$
    \chi_S(x) = (x-\lambda_0)^{n-d}(x-\kappa)^{d+1-\theta} \phi(x),
$$
for some monic integer polynomial $\phi(x)$ of degree $\theta-1$.
\end{lem}
\begin{proof}
Since its multiplicity is greater than $1$, the eigenvalue $\lambda_0$ must be odd~\cite[Theorem 2.2]{GKMS16}.
Denote by $\lambda_1, \dots, \lambda_d$ the eigenvalues of $S$ not equal to $\lambda_0$.
Next, since $\tr S=0$ and $\tr S^2=n(n-1)$, we have
\[
	\sum_{i=1}^{d}\lambda_i = -\lambda_0(n-d) \ \ \ \ \ \text{ and } \ \ \ \ \  \sum_{i=1}^{d}\lambda_i^2 = n(n-1)-\lambda_0^2(n-d).
\]
Combining the above yields
\begin{align}\label{eq:dimd_l0ev}
\sum_{i=1}^{d}(\lambda_i-\kappa)^2=n(n-1)-\lambda_0^2(n-d)+2\kappa \lambda_0 (n-d)+d \kappa^2.
\end{align}
The minimum value of the polynomial $n(n-1)-\lambda_0^2(n-d)+2x \lambda_0 (n-d)+d x^2$ is attained when $x=(d-n)\lambda_0/d$.
Hence, the minimum value of $n(n-1)-\lambda_0^2(n-d)+2x \lambda_0 (n-d)+d x^2$ for $x$ an odd integer is attained when $x=\kappa$.

Let $\mathcal T=n(n-1)-\lambda_0^2(n-d)+2\kappa \lambda_0 (n-d)+d \kappa^2$ and let $\eta=d$.
From \eqref{eq:dimd_l0ev}, we have
$
\sum_{i=1}^{\eta}(\lambda_i-\kappa)^2=\mathcal T.
$
Suppose that $\eta$ satisfies 
$\eta4^{(\eta-\gamma(n))/\eta} > \mathcal T$.
By Lemma~\ref{lem:shiftType}, the characteristic polynomial
$$
\chi_{S-\kappa I}(x)=\chi_S(x+\kappa)=x^{d-\eta}(x-\lambda_0+\kappa)^{n-d}\prod_{i=1}^{\eta}(x-\lambda_i+\kappa)
$$
is weakly type 2 and is type 2 if $n$ is even.
By Lemma~\ref{lem:type_factor}, the polynomial $\prod_{i=1}^{\eta}(x-\lambda_i+\kappa)$ is also weakly type 2, or type 2 if $n$ is even.
In particular, the constant term $\prod_{i=1}^{\eta}(\lambda_i-\kappa)$ is divisible by $2^{\eta-\gamma(n)}$.
Thus, we write $\prod_{i=1}^{\eta}(\lambda_i-\kappa)=2^{\eta-\gamma(n)}\cdot k_{\eta}$ where $k_{\eta}\in \mathbb{Z}$.
Using the inequality of arithmetic and geometric means, we obtain
$$
\eta^{\eta} 4^{\eta-\gamma(n)} > \mathcal T^{\eta} = \left ( \sum\limits_{i=1}^{\eta}(\lambda_i-\kappa)^2\right ) ^{\eta} \geqslant {\eta}^{\eta}\prod_{i=1}^{\eta}(\lambda_i-\kappa)^2 = {\eta}^{\eta} 4^{\eta-\gamma(n)}\cdot k_{\eta}^2.
$$
This implies that $k_{\eta}=0$ and, without loss of generality, we can assume $\lambda_{\eta}=\lambda_d = \kappa$.
Then 
$$
\chi_{S-\kappa I}(x)=x^{d+1-\eta}(x-\lambda_0+\kappa)^{n-d}\prod_{i=1}^{\eta-1}(x-\lambda_i+\kappa).
$$
Furthermore, we have $\sum_{i=1}^{\eta-1}(\lambda_i-\kappa)^2=\mathcal T$ and, by Lemma~\ref{lem:type_factor}, the polynomial $\prod_{i=1}^{\eta-1}(x-\lambda_i+\kappa)$ is weakly type 2, or type 2 if $n$ is even.

Suppose $\theta \leqslant d$.
Note that for $\eta > 0$, the function $\Psi(\eta)=\eta 4^{(\eta-\gamma(n))/\eta}$ is increasing.
Hence, for each integer $\eta$ where $\theta \leqslant \eta \leqslant d$, we have $\eta 4^{(\eta-\gamma(n))/\eta} > \mathcal T$.
Inductively repeat the steps above from $\eta=d$ to $\eta=\theta$.
We obtain $\lambda_{\theta}=\dots=\lambda_d=\kappa$ and thus
$$
\chi_S(x)=(x-\lambda_0)^{n-d}(x-\kappa)^{d+1-\theta} \prod_{i=1}^{\theta-1} (x-\lambda_i),
$$
as required.
\end{proof}

\begin{remark}
It is interesting to note that the extremal case of Lemma~\ref{lem:ev_mult} (when $\theta = 1$) characterises Seidel matrices having precisely two distinct eigenvalues.
Such Seidel matrices correspond to regular two-graphs~\cite{taylor}.
\end{remark}

The next result follows immediately from Lemma~\ref{lem:ev_mult} and Theorem~\ref{thm:smallesteigenvalue}.

\begin{cor}\label{cor:ev_mult}
\begin{enumerate}
Let $S$ be a Seidel matrix.
    \item[(a)] If $S$ corresponds to 75 equiangular lines in $\mathbb{R}^{19}$ then
$
\chi_S(x)=(x+5)^{56}(x-15)^{12}\phi(x),
$
for some monic polynomial $\phi$ of degree 7 in $\mathbb{Z}[x]$.

    \item[(b)] If $S$ corresponds to 95 equiangular lines in $\mathbb{R}^{20}$ then
$
\chi_S(x)=(x+5)^{75}(x-19)^{13}\phi(x),
$
for some monic polynomial $\phi$ of degree 7 in $\mathbb{Z}[x]$.

    \item[(c)] If $S$ corresponds to 29 equiangular lines in $\mathbb{R}^{14}$ then 
$
\chi_S(x)=(x+5)^{15}(x-5)^4 \phi(x),
$
for some monic polynomial $\phi$ of degree 10 in $\mathbb{Z}[x]$.
    
    \item[(d)] If $S$ corresponds to 41 equiangular lines in $\mathbb{R}^{16}$ then 
$
\chi_S(x)=(x+5)^{25}(x-7)^3 \phi(x),
$
for some monic polynomial $\phi$ of degree 13 in $\mathbb{Z}[x]$.
\end{enumerate}
\end{cor}


Let $S$ be a Seidel matrix corresponding to an equiangular line system of cardinality $n$ in $\mathbb R^d$, where $(n,d) \in \{ (29,14), (41,16), (75,19), (95,20) \}$.
By Theorem~\ref{thm:smallesteigenvalue}, in each case, the smallest eigenvalue of $S$ is $\lambda_0 = -5$.
We can determine $\kappa$ and $\theta$ (as defined in Lemma~\ref{lem:ev_mult}) for each $(n,d)$. 
In view of Corollary~\ref{cor:ev_mult}, the next step is to find feasible polynomials for $\phi(x)=\sum_{t=0}^{\theta-1} b_t x^{\theta-1-t}$ where
$$
    \chi_S(x) = (x-\lambda_0)^{n-d}(x-\kappa)^{d+1-\theta} \phi(x).
$$
Obviously $b_0 = 1$.
And we can find $b_1$ and $b_2$, using a basic fact about Seidel matrices: the traces of $S$ and $S^2$ are expressed in terms of $n$ as $\tr S=0$ and $\tr S^2=n(n-1)$.
Hence, using Newton's identities, we have $b_1=\lambda_0(n-d)+\kappa(d+1-\theta)$ and $b_2=\left(b_1^2+\lambda_0^2(n-d)+\kappa^2(d+1-\theta)-n(n-1)\right)/2$.

By Lemma~\ref{lem:shiftType}, since in each case $n$ is odd, the polynomial
$$
\chi_S(x-1)=(x-\lambda_0-1)^{n-d}(x-\kappa-1)^{d+1-\theta}\phi(x-1)
$$
is weakly type 2.
By Lemma~\ref{lem:type_factor}, the polynomial $\phi(x-1)$ is also weakly type 2.
Thus, we need to find all totally-real, integer polynomials $\phi(x)$ with the following properties:
\begin{enumerate}[label=(\roman*)]
\item $b_0=1$, $b_1=\lambda_0(n-d)+\kappa(d+1-\theta)$, and $b_2=\left(b_1^2+\lambda_0^2(n-d)+\kappa^2(d+1-\theta)-n(n-1)\right)/2$,
\item $\phi(x-1)$ is weakly type 2.
\item $(x-\lambda_0)^{n-d}(x-\kappa)^{d+1-\theta} \phi(x)$ belongs to a congruence class in $\mathcal{P}_{n,7}$.\label{itm:iv}
\end{enumerate}

Let $\Phi_{n,d}$ be the set of polynomials satisfying these properties.
We use the polynomial generation algorithm of Section~\ref{sec:pea} to compute $\Phi_{n,d}$ for each $(n,d) \in \{ (29,14), (41,16), (75,19), (95,20) \}$ and we form the set $P_{n,d}$ from elements of $\Phi_{n,d}$ multiplied by $(x-\lambda_0)^{n-d}(x-\kappa)^{d+1-\theta}$.
We list the result of our computations in the following proposition.

\begin{prop}
\label{prop:candpols}
    Let $(n,d) \in \{ (29,14), (41,16), (75,19), (95,20) \}$ and let $S$ be a Seidel matrix corresponding to $n$ equiangular lines in $\mathbb R^d$.
     Then $\chi_S(x) \in P_{n,d}$.
     The set $P_{n,d}$ is described explicitly below.
     
    \begin{itemize}
        \item $P_{75,19}$ consists of the elements of
        \[
        E_{75,19} = \{ (x+5)^{56}(x-13)^4(x-15)^{14}(x-18), \quad  (x+5)^{56}(x-10)(x-15)^{18} \}
        \]
        together with the six polynomials listed in Table~\ref{tab:6polydim19}.
        (Running time: 0.52 seconds.)
        
        \item $P_{95,20}$ consists of the elements of
        \[
        E_{95,20} = \{ (x+5)^{75}(x-17)^4(x-19)^{15}(x-22), \quad (x+5)^{75}(x-14)(x-19)^{19} \}
        \]
        together with the six polynomials listed in Table~\ref{tab:6polydim20}.
        (Running time: 0.61 seconds.)
        
        \item $P_{29,14}$ consists of the elements of 
        
        \[
            E_{29,14} = \left \{ \begin{array}{l}
                (x+5)^{15}(x-5)^{10}(x-7)^2(x^2-11x+16), \\
                (x+5)^{15}(x-3)(x-5)^9(x-7)^2(x^2-13x+32), \\
                (x+5)^{15}(x-5)^{10}(x-7)(x^3-18x^2+93x-128), \\
                (x+5)^{15}(x-3)(x-5)^{11}(x^2-17x+68), \\
                (x+5)^{15}(x-3)^2(x-5)^8(x-7)^2(x^2-15x+52), \\
                (x+5)^{15}(x-3)(x-4)(x-5)^{10}(x-9)^2
            \end{array}
            \right \}
        \]
together with the 25 polynomials listed in Table~\ref{tab:25polydim14}.
(Running time: 69.33 seconds.)

        \item $P_{41,16}$ consists of the elements of
        \[
        E_{41,16} = \{ (x+5)^{25}(x-7)^9(x-9)^4(x-11)(x^2-15x+48), \;  (x+5)^{25}(x-3)(x-7)^6(x-8)(x-9)^8 \}
        \]
         together with the 20 polynomials listed in Table~\ref{tab:20polydim16}.
        (Running time: 395.43 seconds.)
    \end{itemize}
\end{prop}

The running times given in Proposition~\ref{prop:candpols} are for a SageMath implementation~\cite{github} of the polynomial enumeration algorithm running on a modern PC.

\section{Seidel matrices having a prescribed characteristic polynomial}
\label{sec:prescribed}
In this section, we describe a procedure for showing the nonexistence of a Seidel matrix having characteristic polynomial $p(x)$, where $p(x)$ is some fixed polynomial.
		
Let $S$ be a Seidel matrix of order $n$ with characteristic polynomial $\chi_S(x) = p(x)$.
We define $m_S(x) :=\sum_{t=0}^{\delta} a_t x^{\delta-t}$ to be its minimal polynomial of degree $\delta$ and let $\mu_S(x)=\chi_S(x)/m_S(x)$.

\subsection{Interlacing characteristic polynomials}

\label{sec:icp}

Let $i\in \{1,2,\dots,n\}$.
By Lemma~\ref{lem:coeff_of_interlacepoly}, we have
$\chi_{S[i]}(x)=\mu_S(x)\cdot f(x)$ for some polynomial $f(x)=\sum_{t=0}^{\delta-1} b_t x^{\delta-1-t}$ where
$b_0=1$, $b_1=a_1$, and $b_2=a_2+n-1$.
We want to find an exhaustive list of all possibilities for the polynomial $\chi_{S[i]}(x)$.

By Theorem~\ref{thm:basicDivisibility}, the polynomial $\chi_{S[i]}(x-1)$ is weakly type 2 and is type 2 if $n-1$ is even.
By Lemma~\ref{lem:type_factor}, the polynomial $f(x-1)$ is also weakly type 2 and is type 2 if $n-1$ is even.
Thus, we need to find all totally-real, integer polynomials $f(x)=\sum_{t=0}^{\delta-1} b_t x^{\delta-1-t}$ with the following properties:
\begin{enumerate}[label=(\roman*)]
\item $b_0=1$, $b_1=a_1$, $b_2=a_2+n-1$,
\item $f(x)$ interlaces $m_S(x)$,
\item $f(x-1)$ is weakly type 2 and is type 2 if $n-1$ is even,
\item $\mu_{S}(x)\cdot f(x)$ is in a congruence class of $\mathcal{P}_{n-1,7}$, if $n-1$ is odd.
\end{enumerate}

Let $F$ be the set of polynomials satisfying these properties.
We use the polynomial generation algorithm of Section~\ref{sec:pea} to construct $F$.
The set $\mathfrak F := \{ \mu_S(x)\cdot f(x) \; | \; f(x) \in F \}$ is called the set of \textbf{interlacing characteristic polynomials} for $p(x)$.

\subsection{Certificates of infeasibility and interlacing configurations}

The \textbf{coefficient vector} of a polynomial $h(x)=\sum_{t=0}^{n-1} c_{t} x^{n-1-t}$ of degree $n-1$ is defined to be the (row) vector $(c_0,c_1,\dots,c_{n-1})$.
Given a set $H = \{h_1,\dots,h_k\}$ of polynomials each of degree $n - 1$, the \textbf{coefficient matrix} $A(H)$ is defined as the $k \times n$ matrix whose $i$th row is the coefficient vector for $h_i$.
We write $\mathbf x \geqslant \mathbf 0$ to indicate that all entries of the vector $\mathbf x$ are nonnegative.
The polynomial equation \eqref{eqn:thm} can be viewed as a linear system: $\mathbf x^\top A = \mathbf b^\top$, where $\mathbf x \geqslant \mathbf 0$.
Indeed, for a real symmetric matrix $M$ of order $n$, let $X = \{ \chi_{M[i]}(x) \; | \; i \in \{1,\dots,n\}\}$ and let $\mathbf x$ be the vector indexed by elements of $X$ such that the entry indexed by $\mathfrak p (x) \in X$ equals the cardinality of the set $\{ i \in \{1,\dots,n\} \; | \; \chi_{M[i]}(x) = \mathfrak p (x) \}$.
If $A = A(X)$ is the coefficient matrix for the polynomials in $X$ and $\mathbf b$ is the coefficient vector for $\frac{\mathrm{d}}{\mathrm{d}x} \chi_M(x)$, then \eqref{eqn:thm} becomes $\mathbf x^\top A = \mathbf b^\top$.

\begin{theorem}[Farkas' Lemma {\cite{farkas}}]\label{thm:farkas}
Let $A$ be a real $n \times m$ matrix and let $\mathbf{b}\in \mathbb{R}^m$, $\mathbf{x} \in \mathbb{R}^n$.
Then the linear system
$$
\mathbf{x}^\top A=\mathbf{b}^\top,\: \mathbf{x} \geqslant \mathbf{0}
$$
has no solution if and only if the linear system
$$
A\mathbf{y}\geqslant\mathbf{0},\: \mathbf{y}^\top\mathbf{b} <0
$$
has a solution, where $\mathbf{y}\in \mathbb{R}^m$.
\end{theorem}

Theorem~\ref{thm:farkas} allows us to demonstrate that there is no vector $\mathbf x \geqslant \mathbf 0$ satisfying $\mathbf x^\top A = \mathbf b^\top$, by finding a vector $\mathbf y \in \mathbb R^m$ such that $A\mathbf{y}\geqslant\mathbf{0}$ and $\mathbf{y}^\top \mathbf{b} <0$.
We call such a vector $\mathbf y$ a \textbf{certificate of infeasibility} for the linear system $\mathbf{x}^\top A=\mathbf{b}^\top,\: \mathbf{x} \geqslant \mathbf{0}$.\textbf{}

Consider the set $\mathfrak{F} = \{\mathfrak f_1(x),\dots,\mathfrak f_k(x)\}$ of interlacing characteristic polynomials of $S$.
For each $i \in \{1,\dots,k\}$, we have $\mathfrak{f}_i(x) = \mu_S(x)f_i(x)$ for some polynomial $f_i(x) = \sum_{t=0}^{\delta-1} b_{i,t} x^{\delta-1-t}$.
Let $F = \{ f_1(x),$ $\dots,$ $f_k(x) \}$ and let $A=A(F)$.
By Theorem~\ref{thm:sumofsubpoly}, there exist nonnegative integers $n_1,n_2,\dots,n_k$ such that
\begin{equation}\label{eq:sumofsubpoly_dim19}
\mu_S(x)\sum_{i=1}^k n_i \cdot f_i(x)=\frac{\mathrm{d}}{\mathrm{d}x} \chi_S(x).
\end{equation}
Then there exists $g \in \mathbb{Z}[x]$ such that $\frac{\mathrm{d}}{\mathrm{d}x} \chi_S(x)=\mu_S(x)\cdot g(x)$ and \eqref{eq:sumofsubpoly_dim19} can be simplified to $\sum_{i=1}^k n_i\cdot f_i(x)=g(x)$.

Let $\mathbf{n}^\top=(n_1,n_2,\dots,n_k)$.
Then $\mathbf n \geqslant \mathbf 0$ is a solution to the linear system $\mathbf{n}^\top A=\mathbf{g}^\top$ where $\mathbf{g}^\top = (g_0,\dots,g_{\delta-1})$ is the coefficient vector for $g(x)=\sum_{t=0}^{\delta-1} g_{t} x^{\delta-1-t}$.
We call the vector $\mathbf n$ the \textbf{interlacing configuration} for $\mathfrak F$.
Thus, to show that no Seidel matrix $S$ exists having $\chi_S(x) = p(x)$, it suffices to show that there does not exist an interlacing configuration for $\mathfrak F$.
Hence, it suffices to provide a certificate of infeasibility $\mathbf c$ for the linear system above.
We call $\mathbf c$ a \textbf{certificate of infeasibility} for $p(x)$.

Suppose the linear system $\mathbf{n}^\top A=\mathbf{g}^\top$ has at least one nonnegative real solution.
Then a polynomial $\mathfrak f \in \mathfrak F$ is called a \textbf{warranted} interlacing characteristic polynomial if there is no nonnegative solution $\mathbf n$ to the subsystem  $\mathbf{n}^\top A^\prime=\mathbf{g}^\top$, where the matrix $A^\prime$ is obtained from $A$ by removing the row corresponding to the polynomial $\mathfrak f$.
We can show that an interlacing characteristic polynomial is warranted by providing a certificate of infeasibility $\mathbf c$ for this subsystem.
We call such a $\mathbf c$ a \textbf{certificate of warranty} for the interlacing characteristic polynomial $\mathfrak f (x)$.
Equivalently, we have that the one entry of $ A\mathbf c$ that corresponds to $\mathfrak f$ is negative, while the rest of the entries of $A\mathbf c$, which correspond to the entries of $A'\mathbf c$, are nonnegative.

\subsection{Eigenspaces, angles, and compatibility}

In this section, we introduce the notion of compatibility for interlacing characteristic polynomials.
This notion is used in Section~\ref{sec:dim1416}.
	Let $M$ be a real symmetric matrix of order $n$.
		We write $\Lambda(M) = \{\lambda_1,\dots,\lambda_m\}$ for the set of distinct eigenvalues of $M$.
		For each $i \in \{1,\dots,m\}$, denote by $\mathcal{E}(\lambda_i)$ the eigenspace of $\lambda_i$ and let $\{\mathbf{e}_1,\dots,\mathbf{e}_n\}$ be the standard basis of $\mathbb R^n$.
		Denote by $P_\lambda$ the orthogonal projection of $\mathbb R^n$ onto $\mathcal{E}(\lambda)$.
		For a vector $\mathbf v$, we write $\mathbf v(i)$ to denote its $i$th entry.
	\begin{theorem}[Spectral Decomposition Theorem]\label{thm:spec_decomp}
Let $M$ be a real symmetric matrix and let $\Lambda(M)=\{\lambda_1,\dots,\lambda_{m}\}$.
Then
$$
M=\lambda_1 P_{\lambda_1}+\dots+\lambda_{m} P_{\lambda_{m}}.
$$
\end{theorem}

		We write $\alpha_{i,j} = ||P_{\lambda_i} \mathbf{e}_j||$ for all $i \in \{1,\dots,m\}$ and $j \in \{1,\dots,n\}$.
		As is customary, we refer to the numbers $\alpha_{i,j}$ as the \textbf{angles} of $M$.
	
		The next result takes advantage of the fact that the entries of a unit eigenvector of a simple eigenvalue can be expressed in terms of the angles.
		See \cite{tao} for a survey on a related result.
    
    \begin{lem}\label{lem:integrality}
      Let $M$ be an integer symmetric matrix of order $n$ with $\Lambda(M) = \{\lambda_1,\dots,\lambda_m\}$ and angles $\alpha_{i,j}$.
       Suppose that $\lambda_1,\dots,\lambda_{l}$
       are simple eigenvalues and $1 \leqslant l \leqslant m$.
       Define $\displaystyle q(x)=\prod_{i=l+1}^{m}(x-\lambda_i)$ and suppose $q(x) \in \mathbb Z[x]$.
       Then for all $i,j \in \{1,\dots,n\}$, there exist $\varepsilon_2,\dots,\varepsilon_l \in \{ \pm 1 \}$ such that
       \begin{equation}
         \label{eqn:integrality}
         q(\lambda_1)\alpha_{1,i}\alpha_{1,j}+\sum_{k=2}^{l}  q(\lambda_k)\varepsilon_k \alpha_{k,i}\alpha_{k,j} \in \mathbb {Z}.
       \end{equation}
    \end{lem}
    
    \begin{proof}
    Let $\mathbf u_1, \dots, \mathbf u_l$ be unit eigenvectors for $\lambda_1,\dots,\lambda_l$ respectively.
        By Theorem~\ref{thm:spec_decomp},
      $$
      q(M)=q(\lambda_1) \mathbf{u}_1 \mathbf{u}_1^{\top} +\dots + q(\lambda_{\epsilon}) \mathbf{u}_{l} \mathbf{u}_{l}^{\top}.
      $$
      
    Observe
    \begin{equation}
    \label{eqn:eigenvectorentry}
        \alpha_{i,j}=\|P_{\lambda_i}\mathbf{e}_j\|=\|\mathbf{u}_i\mathbf{u}_i^{\top}\mathbf{e}_j\|=\|\mathbf{u}_i(\mathbf{u}_i^{\top}\mathbf{e}_j)\|=\|\mathbf{u}_i\|\cdot|\mathbf{u}_i(j)|=|\mathbf{u}_i(j)|.
    \end{equation}
      Thus, $\mathbf{u}_i(j)=\pm\alpha_{i,j}$.
	  
      Let $D$ be a diagonal matrix with diagonal entries $\pm 1$.
        We have
      \begin{align*}
      D \, q(M) D^{-1}
      &=q(\lambda_1) (D \mathbf{u}_1) (D \mathbf{u}_1)^{\top} +\dots + q(\lambda_{l}) (D \mathbf{u}_{l}) (D \mathbf{u}_{l})^{\top}.
      \end{align*}
      Thus, without loss of generality, we can assume that each entry of $\mathbf u_1$ is nonnegative.
	  The lemma follows since $q(M)$ is an integer matrix.
    \end{proof}
    
    	For a proof of the next result see \cite[(4.2.8)]{crs97} or \cite{GoMK}.
		
		\begin{prop}\label{prop:submatrix}
			Let $M$ be a real symmetric matrix of order $n$ with $\Lambda(M) = \{\lambda_1,\dots,\lambda_m\}$ and angles $\alpha_{ij}$.
			Then, for each $j \in \{1,\dots,n\}$, we have
		$$\chi_{M[j]}(x)=\chi_M(x)\sum^m_{i=1}\dfrac{\alpha^2_{ij}}{x-\lambda_i}.$$
		\end{prop}
    
    By Proposition~\ref{prop:submatrix}, there is a correspondence between each interlacing characteristic polynomial $ \mathfrak f_j(x)$ of $\chi_M(x)$ and the set of angles $\{\alpha_{1,j}, \dots, \alpha_{m,j}\}$.
    Accordingly, we may write $\alpha_{i,\mathfrak f_j}$ to mean the angle $\alpha_{i,j}$.
    Let $\mathfrak f(x)$ and $\mathfrak g(x)$ be interlacing characteristic polynomials of $\chi_M(x)$.
    Suppose the eigenvalues $\lambda_1, \dots, \lambda_l$ are all simple eigenvalues and $1 \leqslant l \leqslant m$.
    In view of Lemma~\ref{lem:integrality}, we call $\mathfrak f$ and $\mathfrak g$ \textbf{compatible} if there exist $\varepsilon_2,\dots,\varepsilon_l \in \{\pm 1 \}$ such that \eqref{eqn:integrality} is satisfied.
    That is, for $q(x)=\prod_{i=l+1}^{m}(x-\lambda_i) \in \mathbb Z[x]$, we have
    \begin{equation*}
         q(\lambda_1)\alpha_{1,\mathfrak f}\alpha_{1,\mathfrak g}+\sum_{k=2}^{l}  q(\lambda_k)\varepsilon_k \alpha_{k,\mathfrak f}\alpha_{k,\mathfrak g} \in \mathbb {Z}.
       \end{equation*}
    
    In Section~\ref{sec:dim1416}, we use repeatedly the following corollary of Lemma~\ref{lem:integrality}. 
    
    \begin{cor}\label{cor:compatible}
      Let $M$ be a real symmetric matrix of order $n$ having at least one simple eigenvalue.
      Let $X$ and $Y$ be principal submatrices of $M$ of order $n-1$.
      Then $\chi_X(x)$ and $\chi_Y(x)$ are compatible.
    \end{cor}

\section{Dimensions 19 and 20}
\label{sec:1920}
In this section we show that $N(19) \leqslant 74$ and $N(20) \leqslant 94$.
We will require the following result of Azarija and Marc~{\cite{Azarija75} and \cite{Azarija95}}.

\begin{theorem}[{\cite{Azarija75} and \cite{Azarija95}}]
\label{thm:SRGs}
    There does not exist a Seidel matrix having as its characteristic polynomial $(x+5)^{56}(x-10)(x-15)^{18}$ or $(x+5)^{75}(x-14)(x-19)^{19}$.
\end{theorem}

\subsection{Nonexistence of 75 equiangular lines in $\mathbb{R}^{19}$}
\label{sec:19}
Here we prove the first main result.

\begin{theorem}
\label{thm:dim19}
$N(19) \leqslant 74$.
\end{theorem}

To prove Theorem~\ref{thm:dim19}, it suffices to show that there does not exist a system of $75$ equiangular lines in $\mathbb R^{19}$.

\begin{lem}\label{lem:6polydim19}
Let $S$ be a Seidel matrix corresponding to $75$ equiangular lines in $\mathbb R^{19}$.
Then $\chi_S(x) \in E_{75,19}$.
\end{lem}

\begin{proof}
By Proposition~\ref{prop:candpols}, we must have $\chi_S(x) \in P_{75,19}$. 
In Table~\ref{tab:6polydim19} below, we provide, for all but two of the polynomials in $P_{75,19}$, a certificate of infeasibility $\mathbf c$.
The only remaining polynomials from $P_{75,19}$ are in $E_{75,19}$.
\end{proof}

\begin{table}[htbp]
    \centering
    \begin{tabular}{l}
          $(x+5)^{56}(x-13)^4(x-14)(x-15)^{12}(x-17)^2$   \\
          $(1876556160, 0, 0, 257776, 18413)$ \\
         \hline
          $(x+5)^{56}(x-14)(x-15)^{14}(x^2-28x+191)^2$   \\
          $(888359172, 0, 0, 135016, 9645)$ \\
         \hline
          $(x+5)^{56}(x-13)^2(x-15)^{15}(x^2-29x+202)$   \\
          $(0, 0, 0, -3203, -1033)$ \\
         \hline
         
          $(x+5)^{56}(x-13)^2(x-15)^{14}(x-17)(x^2-27x+178)$   \\
          $(0, 0, 0, 523203, 123379, 16674)$ \\
         \hline
         
          $(x+5)^{56}(x-11)(x-15)^{16}(x^2-29x+206)$   \\
          $(417044628, 0, 0, 72416, 5571)$ \\
         \hline
         
          $(x+5)^{56}(x-13)(x-15)^{16}(x^2-27x+174)$   \\
          $(0, 0, 0, -3999, -1337)$ \\
         \hline
    \end{tabular}
    \caption{Certificates of infeasibility for each polynomial in $P_{75,19} \backslash E_{75,19}$.  (The computation to find the interlacing characteristic polynomials in each case took less than 0.15 seconds on a modern PC running SageMath~\cite{github}.)}
    \label{tab:6polydim19}
\end{table}

It remains to show that there does not exist a Seidel matrix whose characteristic polynomial is either of the polynomials in $E_{75,19}$.

\begin{lem}\label{lem:dim19_ev18}
There does not exist a Seidel matrix $S$ with 
$$
\chi_S(x)=(x+5)^{56}(x-13)^4(x-15)^{14}(x-18).
$$
\end{lem}
\begin{proof}
Suppose a Seidel matrix $S$ has characteristic polynomial
$
\chi_S(x)=(x+5)^{56}(x-13)^4(x-15)^{14}(x-18).
$
There are four interlacing characteristic polynomials for $\chi_S(x)$.
Each has the form $(x+5)^{55}(x-13)^3(x-15)^{13} f_i(x)$ for $i \in \{1,\dots,4\}$, where 
\begin{align*}
f_1(x)=x^3-41x^2+543x-2319,& \quad f_2(x)=x^3-41x^2+543x-2311,\\
f_3(x)=x^3-41x^2+543x-2303,& \quad
f_4(x)=(x-9)(x-15)(x-17).
\end{align*}
(The computation to find these four interlacing characteristic polynomials took 0.01 seconds on a modern PC running SageMath.)

Now, the interlacing characteristic polynomial $(x+5)^{55}(x-13)^3(x-15)^{13} f_1(x)$ is warranted, with certificate of warranty $(349511, 0, 0, 151)$.
However, this warranted polynomial $(x+5)^{55}(x-13)^3(x-15)^{13} f_1(x)$ has certificate of infeasibility
$(55783250920, 0, 0, 4439925, 330721, 24804)$. 
(The computation to find the $29$ interlacing characteristic polynomials for $(x+5)^{55}(x-13)^3(x-15)^{13} f_1(x)$  took 3.33 seconds on a modern PC running SageMath~\cite{github}.)
Hence $S$ cannot exist.
\end{proof}

Using Lemma~\ref{lem:6polydim19} together with Theorem~\ref{thm:SRGs} and Lemma~\ref{lem:dim19_ev18} we obtain Theorem~\ref{thm:dim19}.

\subsection{Nonexistence of 95 equiangular lines in $\mathbb{R}^{20}$}

In this section we prove the second main result.
It is interesting to note the similarities between this section and Section~\ref{sec:19}.

\begin{theorem}
\label{thm:dim20}
$N(20) \leqslant 94$.
\end{theorem}

To prove Theorem~\ref{thm:dim20}, it suffices to show that there does not exist a system of $95$ equiangular lines in $\mathbb R^{20}$.

\begin{lem}\label{lem:6polydim20}
Let $S$ be a Seidel matrix corresponding to $95$ equiangular lines in $\mathbb R^{20}$.
Then $\chi_S(x) \in E_{95,20}$.
\end{lem}

\begin{proof}
By Proposition~\ref{prop:candpols}, we must have $\chi_S(x) \in P_{95,20}$. 
In Table~\ref{tab:6polydim20} below, we provide, for all but two of the polynomials in $P_{95,20}$, a certificate of infeasibility $\mathbf c$.
The only remaining feasible polynomials from $P_{95,20}$ are in $E_{95,20}$.
\end{proof}

\begin{table}[htbp]
    \centering
    \begin{tabular}{l}
          $(x+5)^{75}(x-17)^4(x-18)(x-19)^{13}(x-21)^2$   \\
          $(14218410144, 0, 0, 889549, 49420)$ \\
         \hline
         
          $(x+5)^{75}(x-18)(x-19)^{15}(x^2-36x+319)^2$   \\
          $(7099110312, 0, 0, 479808, 26657)$ \\
         \hline
         
          $(x+5)^{75}(x-17)^2(x-19)^{16}(x^2-37x+334)$   \\
          $(0, 0, 0, -4529, -1095)$ \\
         \hline
         
          $(x+5)^{75}(x-17)^2(x-19)^{15}(x-21)(x^2-35x+302)$   \\
          $(0, 0, 0, 497653, 88657, 8956)$ \\
         \hline
         
          $(x+5)^{75}(x-15)(x-19)^{17}(x^2-37x+338)$   \\
          $(1784197530, 0, 0, 133358, 7845)$ \\
         \hline
         
          $(x+5)^{75}(x-17)(x-19)^{17}(x^2-35x+298)$   \\
          $(0, 0, 0, -10412, -2587)$ \\
         \hline
    \end{tabular}
    \caption{Certificates of infeasibility for each polynomial in $P_{95,20} \backslash E_{95,20}$.
    (The computation to find the interlacing characteristic polynomials in each case took less than 0.15 seconds on a modern PC running SageMath~\cite{github}.)}
    \label{tab:6polydim20}
\end{table}

It remains to show that there does not exist a Seidel matrix whose characteristic polynomial is either of the polynomials in $E_{95,20}$.

\begin{lem}\label{lem:dim20_ev22}
There does not exist a Seidel matrix $S$ with characteristic polynomial
$$
\chi_S(x)=(x+5)^{75}(x-17)^4(x-19)^{15}(x-22).
$$
\end{lem}
\begin{proof}
Suppose a Seidel matrix $S$ has characteristic polynomial
$
\chi_S(x)=(x+5)^{75}(x-17)^4(x-19)^{15}(x-22).
$
There are four interlacing characteristic polynomials for $\chi_S(x)$.
Each has the form $(x+5)^{74}(x-17)^3(x-19)^{14} f_i(x)$ for $i \in \{1,\dots,4\}$, where 
\begin{align*}
f_1(x)=x^3-53x^2+919x-5211,& \quad
f_2(x)=x^3-53x^2+919x-5203,\\
f_3(x)=x^3-53x^2+919x-5195,& \quad
f_4(x)=(x-13)(x-19)(x-21).
\end{align*}
(The computation to find these four interlacing characteristic polynomials took 0.01 seconds on a modern PC running SageMath~\cite{github}.)

Now, the interlacing characteristic polynomial $(x+5)^{74}(x-17)^3(x-19)^{14} f_1(x)$ is warranted, with certificate of warranty $(4776043, 0, 0, 917)$.
However, this warranted polynomial $(x+5)^{74}(x-17)^3(x-19)^{14} f_1(x)$ has certificate of infeasibility 
$(2282735504746, 0, 0, 79981334, 4595613, 265131)$.
(The computation to find the $29$ interlacing characteristic polynomials for $(x+5)^{55}(x-13)^3(x-15)^{13} f_1(x)$  took 3.42 seconds on a modern PC running SageMath~\cite{github}.)
Hence $S$ cannot exist.
\end{proof}

Using Lemma~\ref{lem:6polydim20} together with Theorem~\ref{thm:SRGs} and Lemma~\ref{lem:dim20_ev22} we obtain Theorem~\ref{thm:dim20}.

\section{Dimensions 14 and 16}

\label{sec:dim1416}

In this section we show that $N(14) = 28$ and $N(16) = 40$.

\subsection{Nonexistence of 29 equiangular lines in $\mathbb{R}^{14}$}

In this section we prove our third main result.

\begin{theorem}
\label{thm:dim14}
$N(14) = 28$.
\end{theorem}

To prove Theorem~\ref{thm:dim14}, since there exist configurations of $28$ equiangular lines in $\mathbb R^{14}$ \cite{GKMS16,lemmens73}, it suffices to show that there does not exist a system of $29$ equiangular lines in $\mathbb R^{14}$.

\begin{lem}\label{lem:25polydim14}
Suppose $S$ is a Seidel matrix corresponding to 29 equiangular lines in $\mathbb{R}^{14}$.
Then $\chi_S(x) \in E_{29,14}$.
\end{lem}

\begin{proof}
By Proposition~\ref{prop:candpols}, we must have $\chi_S(x) \in P_{29,14}$. 
In Table~\ref{tab:25polydim14} (see the appendix), we provide, for all but six of the polynomials in $P_{29,14}$, a certificate of infeasibility $\mathbf c$. 
The only remaining polynomials from $P_{29,14}$ are in $E_{29,14}$.
\end{proof}

It remains to show that there does not exist a Seidel matrix whose characteristic polynomial is in $E_{29,14}$.
We will show that each polynomial in $E_{29,14}$ cannot be the characteristic polynomial of a Seidel matrix.

\begin{lem}\label{lem:dim14last6quadratic16}
There does not exist a Seidel matrix $S$ with characteristic polynomial
$$
\chi_S(x)=(x+5)^{15}(x-5)^{10}(x-7)^2(x^2-11x+16).
$$
\end{lem}
\begin{proof}
Suppose a Seidel matrix $S$ has characteristic polynomial
$$\chi_S(x)=(x+5)^{15}(x-5)^{10}(x-7)^2(x^2-11x+16).$$
Let $(\lambda_1,\dots,\lambda_5)=\left((11-\sqrt{57})/2,(11+\sqrt{57})/2,-5,5,7\right)$ be a $5$-tuple of distinct eigenvalues of $S$.

There are 31 interlacing characteristic polynomials for $S$.
(The computation to find these $31$ interlacing characteristic polynomials took 0.12 seconds on a modern PC running SageMath~\cite{github}.)
The polynomials $\mathfrak f_1(x)$ and $\mathfrak f_2(x)$ given by 
\begin{align*}
\mathfrak f_1(x) &=(x+5)^{14}(x-5)^9(x-7)(x^4-18x^3+96x^2-142x+31),\\
\mathfrak f_2(x) &=(x+5)^{14}(x-5)^9(x-7)^2(x^3-11x^2+19x-1).
\end{align*}
are both warranted, with certificates of warranty $(3936, 0, 0, 29, 0)$ and $(-333696,$ $0,$ $0,$ $-2459,$ $-492)$ respectively.

Using Proposition~\ref{prop:submatrix}, we find
\[
 \begin{bmatrix}
   \alpha_{1,\mathfrak f_1}^2 & \alpha_{2,\mathfrak f_1}^2  \\
   \alpha_{1,\mathfrak f_2}^2 & \alpha_{2,\mathfrak f_2}^2 
 \end{bmatrix}
 = \begin{bmatrix}
    (57+\sqrt{57})/4788 & (57-\sqrt{57})/4788 \;  \\
    (19-\sqrt{57})/456 & (19+\sqrt{57})/456
  \end{bmatrix}.
\]
However, $\mathfrak f_1$ and $\mathfrak f_2$ are not compatible, which contradicts Corollary~\ref{cor:compatible}.
\end{proof}

\begin{lem}\label{lem:dim14last6quadratic32}
There does not exist a Seidel matrix $S$ with characteristic polynomial
$$
\chi_S(x)=(x+5)^{15}(x-3)(x-5)^9(x-7)^2(x^2-13x+32).
$$
\end{lem}
\begin{proof}
Suppose a Seidel matrix $S$ has characteristic polynomial
$$
\chi_S(x)=(x+5)^{15}(x-3)(x-5)^9(x-7)^2(x^2-13x+32).
$$
Let $(\lambda_1,\dots,\lambda_6)=\left(3,(13-\sqrt{41})/2,(13+\sqrt{41})/2,-5,5,7\right)$ be a $6$-tuple of distinct eigenvalues of $S$.

There are 19 interlacing characteristic polynomials for $S$.
(The computation to find these $19$ interlacing characteristic polynomials took 0.62 seconds on a modern PC running SageMath~\cite{github}.)
The polynomials $\mathfrak f_1(x)$ and $\mathfrak f_2(x)$ given by 
\begin{align*}
\mathfrak f_1(x) &=(x+5)^{14}(x-3)(x-5)^8(x-7)(x^4-20x^3+126x^2-260x+73)\\
\mathfrak f_2(x) &=(x+5)^{14}(x-3)(x-5)^8(x-7)^2(x^3-13x^2+35x-7)
\end{align*}
are both warranted, with certificates of warranty $(0, 0, 0, -845, -703, -235)$ and 
$(-2729736,$ $0,$ $0,$ $-6261,$ $-1593,$ $-433)$ respectively.

Using Proposition~\ref{prop:submatrix}, we find
\[
 \begin{bmatrix}
   \alpha_{1,\mathfrak f_1}^2 & \alpha_{2,\mathfrak f_1}^2 & \alpha_{3,\mathfrak f_1}^2  \\
   \alpha_{1,\mathfrak f_2}^2 & \alpha_{2,\mathfrak f_2}^2 & \alpha_{3,\mathfrak f_2}^2
 \end{bmatrix}
 = \begin{bmatrix}
   0 &  (1107+133\sqrt{41})/50020 & (1107-133\sqrt{41})/50020  \\
   0 & (205-7\sqrt{41})/5002 & (205+7\sqrt{41})/5002
  \end{bmatrix}.
\]
However, $\mathfrak f_1$ and $\mathfrak f_2$ are not compatible, which contradicts Corollary~\ref{cor:compatible}.
\end{proof}

\begin{lem}\label{lem:dim14last6cubic}
There does not exist a Seidel matrix $S$ with characteristic polynomial
$$
\chi_S(x)=(x+5)^{15}(x-5)^{10}(x-7)(x^3-18x^2+93x-128).
$$
\end{lem}
\begin{proof}
Suppose a Seidel matrix $S$ has characteristic polynomial
$$
\chi_S(x)=(x+5)^{15}(x-5)^{10}(x-7)(x^3-18x^2+93x-128).
$$
Let $(\lambda_1,\dots,\lambda_6)=(\rho_1,\rho_2,\rho_3,7,-5,5)$ be a $6$-tuple of distinct eigenvalues of $S$, where $\rho_1<\rho_2<\rho_3$ are the zeros of $x^3-18x^2+93x-128$.

There are 41 interlacing characteristic polynomials for $S$.
(The computation to find these $41$ interlacing characteristic polynomials took 1.03 seconds on a modern PC running SageMath~\cite{github}.)
The polynomials $\mathfrak f_1(x)$ and $\mathfrak f_2(x)$ given by 
\begin{align*}
\mathfrak f_1(x)&=(x+5)^{14}(x-5)^9(x^5-25x^4+222x^3-830x^2+1137x-249) \\
\mathfrak f_2(x)&=(x+5)^{14}(x-5)^9(x-7)(x^2-6x+1)(x^2-12x+23)
\end{align*}
are both warranted, with certificates of warranty $(53248032,$ $0,$ $0,$ $80553,$ $12395,$ $2066)$ and $(0, 0, 0, 416, 329, 101)$ respectively.

Let $\sigma_1<\sigma_2<\sigma_3$ be the zeros of $2932848x^3-374976x^2+9513x-53$ and let $\tau_1<\tau_2<\tau_3$ be the zeros of $61101x^3-5022x^2+126x-1$.
Using Proposition~\ref{prop:submatrix}, we find
\[
 \begin{bmatrix}
   \alpha_{1,\mathfrak f_1}^2 & \alpha_{2,\mathfrak f_1}^2 & \alpha_{3,\mathfrak f_1}^2 & \alpha_{4,\mathfrak f_1}^2  \\
   \alpha_{1,\mathfrak f_2}^2 & \alpha_{2,\mathfrak f_2}^2 & \alpha_{3,\mathfrak f_2}^2 & \alpha_{4,\mathfrak f_2}^2
 \end{bmatrix}
 = \begin{bmatrix}
    \sigma_2 & \sigma_3 & \sigma_1 & 1/12  \\
    \tau_1 & \tau_2 & \tau_3 & 0
  \end{bmatrix}.
\]
However, $\mathfrak f_1$ and $\mathfrak f_2$ are not compatible, which contradicts Corollary~\ref{cor:compatible}.
\end{proof}

\begin{lem}\label{lem:dim14last6quadratic68}
There does not exist a Seidel matrix $S$ with characteristic polynomial
$$
\chi_S(x)=(x+5)^{15}(x-3)(x-5)^{11}(x^2-17x+68).
$$
\end{lem}
\begin{proof}
Suppose a Seidel matrix $S$ has characteristic polynomial
$$\chi_S(x)=(x+5)^{15}(x-3)(x-5)^{11}(x^2-17x+68).$$
Let $(\lambda_1,\dots,\lambda_5)=\left( 3,(17-\sqrt{17})/2,(17+\sqrt{17})/2,-5,5 \right)$ be a $5$-tuple of distinct eigenvalues of $S$.

There are 23 interlacing characteristic polynomials for $S$.
(The computation to find these $23$ interlacing characteristic polynomials took 0.15 seconds on a modern PC running SageMath~\cite{github}.)
The polynomial $\mathfrak f_1(x)$ given by 
\begin{align*}
(x+5)^{14}(x-5)^{10}(x^4-20x^3+122x^2-228x+29)
\end{align*}
is warranted, with certificate of warranty $(-256444, 0, 0, -1143, -190)$.

Using Proposition~\ref{prop:submatrix}, we find
\[
 \begin{bmatrix}
   \alpha_{1,\mathfrak f_1}^2 & \alpha_{2,\mathfrak f_1}^2 & \alpha_{3,\mathfrak f_1}^2
 \end{bmatrix}
 = \begin{bmatrix}
    1/26 & (901-171\sqrt{17})/39338 & (901+171\sqrt{17})/39338
  \end{bmatrix}.
\]
Out of the $22$ other interlacing characteristic polynomials of $S$, only
\[
  \mathfrak f_2(x) = (x+5)^{14}(x-5)^{11}(x^3-15x^2+47x-1)
\]
is compatible with $\mathfrak f_1(x)$.
However, $\mathbf n^\top A(\{\mathfrak f_1,\mathfrak f_2\}) = \mathbf b^\top$, where $\mathbf b$ is the coefficient vector for $\frac{\mathrm{d}}{\mathrm{d}x} \chi_S(x)$,  has certificate of infeasibility $(0, 0, 0, -7, -46)$.
This contradicts Theorem~\ref{thm:sumofsubpoly}.
\end{proof}

\begin{lem}\label{lem:dim14last6quadratic52}
There does not exist a Seidel matrix $S$ with characteristic polynomial
$$
\chi_S(x)=(x+5)^{15}(x-3)^2(x-5)^8(x-7)^2(x^2-15x+52).
$$
\end{lem}
\begin{proof}
Suppose a Seidel matrix $S$ has characteristic polynomial
$$
\chi_S(x)=(x+5)^{15}(x-3)^2(x-5)^8(x-7)^2(x^2-15x+52).
$$
Let $(\lambda_1,\dots,\lambda_6)=\left((15-\sqrt{17})/2,(15+\sqrt{17})/2,-5,3,5,7\right)$ be a $6$-tuple of distinct eigenvalues of $S$.

There are 16 interlacing characteristic polynomials for $S$.
(The computation to find these $16$ interlacing characteristic polynomials took 0.46 seconds on a modern PC running SageMath~\cite{github}.)
The polynomials $\mathfrak f_1(x)$ and $\mathfrak f_2(x)$ given by 
\begin{align*}
\mathfrak f_1(x)&=(x+5)^{14}(x-3)^2(x-5)^7(x-7)^2(x^3-15x^2+55x-17) \\
\mathfrak f_2(x)&=(x+5)^{14}(x-3)(x-5)^9(x-7)(x^3-15x^2+51x+11)
\end{align*}
are both warranted, with certificates of warranty $(0, 0, 0, 708, 507, 154)$ and $(26302558,$ $0,$ $0,$ $39525,$ $6783,$ $1196)$ respectively.

Out of these 16 interlacing characteristic polynomials, only $7$ are compatible with $\mathfrak f_1(x)$.
Let $\mathfrak F_1 = \{\mathfrak f_1(x),\mathfrak f_2(x),\dots,\mathfrak f_7(x)\}$ consist of these $7$ polynomials, where
\begin{align*}
\mathfrak f_3(x)&=(x+5)^{14}(x-3)(x-5)^9(x-7)(x^3-15x^2+51x+19),\\ 
\mathfrak f_4(x)&=(x+5)^{14}(x-3)(x-5)^8(x-7)(x^4-20x^3+126x^2-236x-79).
\end{align*}

There is a unique nonnegative solution $\mathbf n = (20,7,1,1,0,0,0)$ to the equation $\mathbf n^\top A(\mathfrak F_1) = \mathbf b^\top$, 
where $\mathbf b$ is the coefficient vector for $\frac{\mathrm{d}}{\mathrm{d}x} \chi_S(x)$.
Note that
$$
(S^2-25I)(S-3I)(S-7I)=\frac{1099-275\sqrt{17}}{2}\mathbf{u}\mathbf{u}^{\top}+\frac{1099+275\sqrt{17}}{2}\mathbf{v}\mathbf{v}^{\top}
$$
where $\mathbf{u}$ and $\mathbf{v}$ are unit eigenvectors corresponding to the eigenvalues $\lambda_1$ and $\lambda_2$, respectively.

We will attempt to construct $\mathbf{u}$ and $\mathbf{v}$ from the angles of the corresponding interlacing characteristic polynomials of $S$ for $\lambda_1$ and $\lambda_2$.
Using Proposition~\ref{prop:submatrix}, we find that 
\[
 \begin{bmatrix}
    \alpha_{1,\mathfrak f_1}^2 & \alpha_{2,\mathfrak f_1}^2  \\
    \alpha_{1,\mathfrak f_2}^2 & \alpha_{2,\mathfrak f_2}^2  \\
    \alpha_{1,\mathfrak f_3}^2 & \alpha_{2,\mathfrak f_3}^2 \\
    \alpha_{1,\mathfrak f_4}^2 & \alpha_{2,\mathfrak f_4}^2
 \end{bmatrix}
 = \begin{bmatrix}
    (51-\sqrt{17})/1292 & (51+\sqrt{17})/1292 \\
    (221+21\sqrt{17})/{20672} & (221-21\sqrt{17})/{20672}  \\
    (51-\sqrt{17})/1292 & (51+\sqrt{17})/1292 \\
    9(221+21\sqrt{17})/{20672} & 9(221-21\sqrt{17})/{20672}
  \end{bmatrix}.
\]
Set $\beta = \sqrt{(51-\sqrt{17})/1292}$ and $\overline \beta = \sqrt{(51+\sqrt{17})/1292}$.
By \eqref{eqn:eigenvectorentry}, the vector $\mathbf{u}$ has 20+1=21 entries equal to $\pm\beta$, 7 equal to $\beta(\sqrt{17}+1)/{8}$, and one equal to $\pm3\beta(\sqrt{17}+1)/{8}$.
Similarly, the entries of $\mathbf{v}$ are 21 of $\pm\overline \beta$, 7 of $\pm\overline \beta(\sqrt{17}-1)/{8}$, and one of $\pm3\overline \beta(\sqrt{17}-1)/{8}$.
Without loss of generality, we can assume that all entries of $\mathbf{u}$ are nonnegative:

$$\mathbf{u}^{\top}=\displaystyle \big(\underbrace{\beta,\dots,\beta}_{20},\underbrace{\beta(\sqrt{17}+1)/{8},\dots,\beta(\sqrt{17}+1)/{8}}_{7},\beta,
3\beta(\sqrt{17}+1)/{8}\big).$$

Suppose $\mathbf{v}(1)=\overline \beta$.
Since the entries of $(S^2-25I)(S-3I)(S-7I)$ are integers, we have that $\mathbf{v}(21)=\dots=\mathbf{v}(27)=-\overline \beta(\sqrt{17}-1)/{8}$ and $\mathbf{v}(29)=-3\overline \beta(\sqrt{17}-1)/{8}$.
Consequently, we must also have that $\mathbf{v}(2)=\dots=\mathbf{v}(20)=\mathbf{v}(28)=\overline \beta$.
Hence
\begin{align*}
\mathbf{v}^{\top}=\big(\underbrace{\overline \beta,\dots,\overline \beta}_{20},\underbrace{-\overline \beta(\sqrt{17}-1)/{8},\dots,-\overline \beta(\sqrt{17}-1)/{8}}_7,\overline \beta,-3\overline \beta(\sqrt{17}-1)/{8}\big).
\end{align*}
However, this implies $\mathbf{u}^\top \mathbf{v}=17/\sqrt{646}$, which is a contradiction since $\mathbf{u}$ and $\mathbf{v}$ must be orthogonal.

Similarly, if $\mathbf{v}(1)=-\overline \beta$ then, using the above reasoning, it follows that $\mathbf{u}^\top \mathbf{v}=-17/\sqrt{646}$, which is again a contradiction.
Therefore, there is no Seidel matrix with characteristic polynomial $(x+5)^{15}(x-3)^2(x-5)^8(x-7)^2(x^2-15x+52)$.
\end{proof}

Lastly, to prove that there does not exist a Seidel matrix $S$ with characteristic polynomial 
$$
\chi_S(x)=(x+5)^{15} (x-3)(x-4)(x-5)^{10}(x-9)^2 \in E_{29,14},
$$ 
we first prove an intermediate nonexistence result.

\begin{lem}\label{lem:kill28}
  There does not exist a Seidel matrix $S$ with characteristic polynomial
    $$
    \chi_{S}(x)=(x+5)^{14}(x-5)^9(x-9)(x^2-4x-1)(x^2-12x+31).
    $$
\end{lem}

\begin{proof}
Suppose a Seidel matrix $S$ has characteristic polynomial
  $$
  \chi_{S}(x)=(x+5)^{14}(x-5)^9(x-9)(x^2-4x-1)(x^2-12x+31).
  $$

We find that there are 124 interlacing characteristic polynomials.
(The computation to find these $124$ interlacing characteristic polynomials took 450.77 seconds on a modern PC running SageMath~\cite{github}.)
Out of these 124 interlacing characteristic polynomials, one is warranted:
 $$
  \mathfrak f_1(x) = (x+5)^{13}(x-5)^8(x^6-25x^5+224x^4-842x^3+1065x^2+387x-554),
  $$
 with certificate of warranty $(0, 0, 0, -7130, -5303, -1486, -344)$.
 
Let $\mathfrak F_1$ be the set of interlacing characteristic polynomials that are compatible with $\mathfrak f_1(x)$.
Then $\mathfrak F_1$ consists of the following five polynomials:
  \begin{align*}
  &(x+5)^{13}(x-5)^8(x^6-25x^5+224x^4-842x^3+1065x^2+387x-554),\\
  &(x+5)^{13}(x-5)^8(x^6-25x^5+224x^4-830x^3+877x^2+1239x-1742),\\
  &(x+5)^{13}(x-1)(x-5)^9(x-9)(x^3-10x^2+15x+34),\\
  &(x+5)^{13}(x-5)^9(x^2-8x-1)(x^3-12x^2+29x+14),\\
  &(x+5)^{13}(x-5)^8(x^3-11x^2+15x+59)(x^3-14x^2+55x-58).
  \end{align*}

There is a unique solution $\mathbf n = (39/2,2,9/4,1/4,4)$ to the equation $\mathbf n^\top A(\mathfrak F_1) = \mathbf b^\top$, 
where $\mathbf b$ is the coefficient vector for $\frac{\mathrm{d}}{\mathrm{d}x} \chi_S(x)$.
However, the coefficients of $\mathbf n$ are not all integers, which contradicts Theorem~\ref{thm:sumofsubpoly}.

\end{proof}

\begin{lem}\label{lem:dim14last6allint}
There does not exist a Seidel matrix $S$ with characteristic polynomial
$$
\chi_S(x)=(x+5)^{15}(x-3)(x-4)(x-5)^{10}(x-9)^2.
$$
\end{lem}
\begin{proof}
Suppose a Seidel matrix $S$ has characteristic polynomial 
$$\chi_S(x)=(x+5)^{15}(x-3)(x-4)(x-5)^{10}(x-9)^2.$$
Let $(\lambda_1,\dots,\lambda_6)=(3,4,-5,5,9)$ be a $5$-tuple of distinct eigenvalues of $S$.

There are six interlacing characteristic polynomials for $S$:
\begin{align*}
\mathfrak f_1(x) &= (x+5)^{14}(x-5)^9(x-9)(x^2-4x-1)(x^2-12x+31),\\
\mathfrak f_2(x) &= (x+5)^{14}(x-3)(x-5)^{10}(x-9)(x^2-8x-1),\\
\mathfrak f_3(x) &= (x+5)^{14}(x-3)(x-5)^9(x-9)(x^3-13x^2+39x-11),\\
\mathfrak f_4(x) &= (x+5)^{14}(x-3)(x-5)^9(x-9)(x^3-13x^2+39x-3),\\
\mathfrak f_5(x) &= (x+5)^{14}(x-5)^{10}(x-9)(x^3-11x^2+23x+11),\\
\mathfrak f_6(x) &= (x+5)^{14}(x-5)^{10}(x-9)(x^3-11x^2+23x+19).
\end{align*}
(The computation to find these six interlacing characteristic polynomials took 0.05 seconds on a modern PC running SageMath~\cite{github}.)

By Lemma~\ref{lem:kill28}, the interlacing characteristic polynomial $\mathfrak f_1(x)$ cannot be a characteristic polynomial of a principal submatrix of $S$.
Thus, we only consider interlacing configurations whose entry corresponding to $\mathfrak f_1(x)$ is $0$.
This leaves us with three possibilities:
\begin{align*}
(0,0,24,2,0,3), \quad \quad (0,0,25,0,2,2), \quad \quad (0,1,25,0,0,3).
\end{align*}
Using Proposition~\ref{prop:submatrix}, we find that 
\[
 \begin{bmatrix}
    \alpha_{1,\mathfrak f_2}^2 & \alpha_{2,\mathfrak f_2}^2  \\
    \alpha_{1,\mathfrak f_3}^2 & \alpha_{2,\mathfrak f_3}^2  \\
    \alpha_{1,\mathfrak f_4}^2 & \alpha_{2,\mathfrak f_4}^2  \\
    \alpha_{1,\mathfrak f_5}^2 & \alpha_{2,\mathfrak f_5}^2  \\
    \alpha_{1,\mathfrak f_6}^2 & \alpha_{2,\mathfrak f_6}^2
 \end{bmatrix}
 = \begin{bmatrix}
    0 &  17/45 \\
    0 & 1/45  \\
    0 & 1/5 \\
    1/6 & 1/5 \\
     1/3 & 1/45
  \end{bmatrix}.
\]
Now let $T=(S^2-25I)(S-9I)$.
Then, by Theorem~\ref{thm:spec_decomp}, we have
$$T=96\mathbf{u}\mathbf{u}^{\top}+45\mathbf{v}\mathbf{v}^{\top},$$
where $\mathbf{u}$ and $\mathbf{v}$ are unit eigenvectors of $S$ corresponding to eigenvalues 3 and 4, respectively.
Moreover, $\mathbf{u}$ and $\mathbf{v}$ must be orthogonal.
Depending on the three interlacing configurations, by \eqref{eqn:eigenvectorentry}, the entries of $\mathbf{u}$ are in $\left \{0, \pm 1/\sqrt{6},\pm 1/\sqrt{3} \right \}$ and the entries of $\mathbf{v}$ are in $\left \{\pm\sqrt{17/45}, \pm 1/\sqrt{45}, \pm 1/\sqrt{5} \right \}$.

If the interlacing configuration is $(0,0,24,2,0,3)$ then, without loss of generality, we have
$$
\mathbf{v}^{\top}=\bigg(\underbrace{\frac{1}{\sqrt{45}},\dots,\frac{1}{\sqrt{45}}}_{24},\frac{1}{\sqrt{5}},\frac{1}{\sqrt{5}},\frac{1}{\sqrt{45}},\frac{1}{\sqrt{45}},\frac{1}{\sqrt{45}}\bigg),
$$
while
$$
\mathbf{u}^{\top}=\bigg(\underbrace{0,\dots,0}_{26},\pm\frac{1}{\sqrt{3}},\pm\frac{1}{\sqrt{3}},\pm\frac{1}{\sqrt{3}}\bigg).
$$
However, it is not possible to have $\mathbf{u}\cdot\mathbf{v}=0$.
The argument is similar for the interlacing configuration $(0,1,25,0,0,3)$.

Finally, if the  interlacing configuration is $(0,0,25,0,2,2)$ then, without loss of generality,
$$
\mathbf{v}^{\top}=\bigg(\underbrace{\frac{1}{\sqrt{45}},\dots,\frac{1}{\sqrt{45}}}_{25},\frac{1}{\sqrt{5}},\frac{1}{\sqrt{5}},\frac{1}{\sqrt{45}},\frac{1}{\sqrt{45}}\bigg).
$$
Then, since $\mathbf u$ is orthogonal to $\mathbf v$, we must have
$$
\mathbf{u}^{\top} \in \left \{ \pm \bigg(\underbrace{0,\dots,0}_{25},\frac{1}{\sqrt{6}},-\frac{1}{\sqrt{6}},\frac{1}{\sqrt{3}},-\frac{1}{\sqrt{3}}\bigg), \pm \bigg(\underbrace{0,\dots,0}_{25},\frac{1}{\sqrt{6}},-\frac{1}{\sqrt{6}},-\frac{1}{\sqrt{3}},\frac{1}{\sqrt{3}}\bigg) \right \}.
$$
In each case, however, the matrix
$
T=96\mathbf{u}\mathbf{u}^{\top}+45\mathbf{v}\mathbf{v}^{\top}
$
is not an integer matrix.
We thereby arrive at a contradiction.
\end{proof}

Using Lemma~\ref{lem:25polydim14} together with Lemmas~\ref{lem:dim14last6quadratic16},~\ref{lem:dim14last6quadratic32},~\ref{lem:dim14last6cubic},~\ref{lem:dim14last6quadratic68},~\ref{lem:dim14last6quadratic52}, and~\ref{lem:dim14last6allint} we obtain Theorem~\ref{thm:dim14}.

\subsection{Nonexistence of 41 equiangular lines in $\mathbb{R}^{16}$}

In this section we prove our final main result.

\begin{theorem}
\label{thm:dim16}
$N(16) \leqslant 40$.
\end{theorem}

To prove Theorem~\ref{thm:dim16}, since there exist configurations of $40$ equiangular lines in $\mathbb R^{16}$ \cite{GKMS16,lemmens73}, it suffices to show that there does not exist a system of $41$ equiangular lines in $\mathbb R^{16}$.

\begin{lem}\label{lem:20polydim16}
Suppose $S$ is a Seidel matrix corresponding to 41 equiangular lines in $\mathbb{R}^{16}$.
Then $\chi_S(x) \in E_{41,16}$.
\end{lem}

\begin{proof}
By Proposition~\ref{prop:candpols}, we must have $\chi_S(x) \in P_{41,16}$.
In Table~\ref{tab:20polydim16} (see the appendix), we provide, for all but two of the polynomials in $P_{41,16}$, a certificate of infeasibility $\mathbf c$.
The only remaining polynomials from $P_{41,16}$ are in $E_{41,16}$.
\end{proof}

It remains to show that there does not exist a Seidel matrix whose characteristic polynomial is either of the polynomials in $E_{41,16}$.

\begin{lem}\label{lem:dim16_last2_quadratic}
There does not exist a Seidel matrix of $S$ with characteristic polynomial
$$
\chi_S(x)=(x+5)^{25}(x-7)^9(x-9)^4(x-11)(x^2-15x+48).
$$
\end{lem}

\begin{proof}
Suppose a Seidel matrix $S$ has characteristic polynomial
$$
\chi_S(x)=(x+5)^{25}(x-7)^9(x-9)^4(x-11)(x^2-15x+48).
$$
Let $(\lambda_1,\dots,\lambda_6)=\left((15-\sqrt{33})/2,(15+\sqrt{33})/2,11,-5,7,9\right)$ be a $6$-tuple of distinct eigenvalues of $S$.

There are 18 interlacing characteristic polynomials for $S$.
(The computation to find these $18$ interlacing characteristic polynomials took 4.08 seconds on a modern PC running SageMath~\cite{github}.)
The polynomials $\mathfrak f_1(x)$, $\mathfrak f_2(x)$, and $\mathfrak f_3(x)$ given by 
\begin{align*}
\mathfrak f_1(x) &=(x+5)^{24}(x-5)(x-7)^8(x-9)^3(x-11)(x^3-21x^2+131x-215),\\
\mathfrak f_2(x) &=(x+5)^{24}(x-5)(x-7)^8(x-9)^4(x-11)(x^2-12x+23),\\
\mathfrak f_3(x) &=(x+5)^{24}(x-7)^8(x-9)^4(x^4-28x^3+270x^2-1028x+1281)
\end{align*}
are each warranted, with certificates of warranty $(0, 0, 0, 0, 124, 113)$, $(8817840108,$ $0,$ $0,$ $3461627,$ $336081,$ $33608)$, and $(0, 0, 0, 11405, 5147, 1289)$, respectively.

Using Proposition~\ref{prop:submatrix}, we find
\[
 \begin{bmatrix}
   \alpha_{1,\mathfrak f_1}^2 & \alpha_{2,\mathfrak f_1}^2 & \alpha_{3,\mathfrak f_1}^2  \\
   \alpha_{1,\mathfrak f_2}^2 & \alpha_{2,\mathfrak f_2}^2 & \alpha_{3,\mathfrak f_2}^2  \\
   \alpha_{1,\mathfrak f_3}^2 & \alpha_{2,\mathfrak f_3}^2 & \alpha_{3,\mathfrak f_3}^2  
 \end{bmatrix}
 = \begin{bmatrix}
    (77+9\sqrt{33})/4884 & (77-9\sqrt{33})/4884 & 0  \\
    (693-67\sqrt{33})/9768 & (693+67\sqrt{33})/9768 & 0\\
    (121-7\sqrt{33})/6512 & (121+7\sqrt{33})/6512 & 1/16
  \end{bmatrix}.
\]
However, $\mathfrak f_1$ and $\mathfrak f_2$ are not compatible, which contradicts Corollary~\ref{cor:compatible}.
\end{proof}

\begin{lem}\label{lem:dim16_last2_allint}
There does not exist a Seidel matrix of $S$ with characteristic polynomial
$$
\chi_S(x)=(x+5)^{25}(x-3)(x-7)^6(x-8)(x-9)^8.
$$
\end{lem}

\begin{proof}
Suppose a Seidel matrix $S$ has characteristic polynomial
$$
\chi_S(x)=(x+5)^{25}(x-3)(x-7)^6(x-8)(x-9)^8.
$$
Let $(\lambda_1,\dots,\lambda_5)=(3,8,-5,7,9)$ be a $5$-tuple of distinct eigenvalues of $S$.

There are five interlacing characteristic polynomials for $S$.
(The computation to find these five interlacing characteristic polynomials took 0.06 seconds on a modern PC running SageMath~\cite{github}.)
The polynomials $\mathfrak f_1(x)$ and $\mathfrak f_2(x)$ given by 
\begin{align*}
\mathfrak f_1(x) &=(x+5)^{24}(x-7)^5(x-9)^7(x^2-10x+17)(x^2-12x+31),\\
\mathfrak f_2(x) &=(x+5)^{24}(x-7)^6(x-9)^7(x^3-15x^2+63x-57)
\end{align*}
are each warranted, with certificates of warranty $(-133896, 0, 0, -304, -43)$ and $(-6402648,$ $0,$ $0,$ $-14059,$ $-1562)$, respectively.

Using Proposition~\ref{prop:submatrix}, we find
\[
 \begin{bmatrix}
   \alpha_{1,\mathfrak f_1}^2 & \alpha_{2,\mathfrak f_1}^2  \\
   \alpha_{1,\mathfrak f_2}^2 & \alpha_{2,\mathfrak f_2}^2 
 \end{bmatrix}
 = \begin{bmatrix}
    1/60 & 1/65 \\
    1/10 & 1/65
  \end{bmatrix}.
\]
However, $\mathfrak f_1$ and $\mathfrak f_2$ are not compatible, which contradicts Corollary~\ref{cor:compatible}.
\end{proof}

Using Lemma~\ref{lem:20polydim16} together with Lemmas~\ref{lem:dim16_last2_quadratic} and~\ref{lem:dim16_last2_allint} we obtain Theorem~\ref{thm:dim16}.

\section{A concluding remark}

    As shown in Table~\ref{tab:equi}, the smallest $d$ for which $N(d)$ is not known is now $d=17$.
    One can apply the techniques of this paper to enumerate the characteristic polynomials for Seidel matrices that potentially correspond to $49$ equiangular lines in $\mathbb R^{17}$.
    In this case, we find that $P_{49,17}$ consists of 194 polynomials.
    Out of these 194 polynomials, two have no interlacing characteristic polynomials and 158 have certificates of infeasibility.
    This leaves us with $34$ polynomials that potentially correspond to an equiangular line system of cardinality $49$ in $\mathbb R^{17}$.
    However, not all of these $34$ polynomials has a warranted interlacing characteristic polynomial.
    In order to deal with such polynomials, heavier computations are required and it is necessary to develop new methods in order to deduce whether or not there exist corresponding Seidel matrices.
    Thus, we leave the consideration of these $34$ polynomials to a future paper.

\appendix 

\section{Tables of certificates of infeasibility}

In Table~\ref{tab:25polydim14}, we list all but six of the polynomials in the set $P_{29,14}$ together with their certificates of infeasibility. 

\begin{longtable}{l}

\endfirsthead

\multicolumn{1}{c}%
{{\bfseries \tablename\ \thetable{} -- continued from previous page}} \\ \hline 
\endhead

\hline \multicolumn{1}{r}{{Continued on next page}} \\ \hline
\endfoot

\endlastfoot

          $(x+5)^{15}(x-4)(x-5)^{12}(x-11)$   \\
          $(423, 0, 0, 44)$ \\
         \hline
         
          $(x+5)^{15}(x-5)^8(x-7)^4(x^2-7x+4)$   \\
          $(-30576, 0, 0, -631, -126)$ \\
         \hline
         
          $(x+5)^{15}(x-3)^3(x-4)(x-5)^4(x-7)^6$   \\
          $(-11155, 0, 0, -114, 0)$ \\
         \hline
         
          $(x+5)^{15}(x-3)^3(x-5)^6(x-7)^4(x-8)$   \\
          $(87568, 0, 0, 458, 51)$ \\
         \hline
         
          $(x+5)^{15}(x-5)^{11}(x^3-20x^2+119x-188)$   \\
          $(-544242, 0, 0, -2563, -366)$ \\
         \hline
         
          $(x+5)^{15}(x-3)(x-5)^{10}(x-7)(x^2-15x+48)$   \\
          $(-34793172, 0, 0, -59189, -12969, -2820)$ \\
         \hline
         
          $(x+5)^{15}(x-3)(x-5)^8(x-7)^3(x^2-11x+20)$   \\
          $(1656424, 0, 0, 4712, 1084, 361)$ \\
         \hline
         
          $(x+5)^{15}(x-3)(x-5)^7(x-7)^4(x^2-9x+12)$   \\
          $(328314, 0, 0, 971, 0, 0)$ \\
         \hline
         
          $(x+5)^{15}(x-3)(x-5)^6(x-7)^5(x^2-7x+8)$   \\
          $(0, 0, 0, 67, 75, 76)$ \\
         \hline
         
          $(x+5)^{15}(x-5)^{10}(x-9)(x^3-16x^2+75x-92)$   \\
          $(0, 0, 0, 479, 395, 112)$ \\
         \hline
         
          $(x+5)^{15}(x-5)^9(x-7)^2(x^3-16x^2+71x-64)$   \\
          $(3452532, 0, 0, 6816, 1105, 184)$ \\
         \hline
         
          $(x+5)^{15}(x-5)^8(x-7)^3(x^3-14x^2+53x-44)$   \\
          $(0, 0, 0, -157, -173, -52)$ \\
         \hline
         
          $(x+5)^{15}(x-3)(x-5)^8(x-8)(x^2-12x+31)^2$   \\
          $(-6439670, 0, 0, -11608, -1528, -217)$ \\
         \hline
         
          $(x+5)^{15}(x-5)^{10}(x^4-25x^3+219x^2-779x+928)$   \\
          $(0, 0, 0, -339, -265, -71)$ \\
         \hline
         
          $(x+5)^{15}(x-4)(x-5)^9(x-7)(x^3-19x^2+107x-169)$   \\
          $(0, 0, 0, 8141, 3743, 676, 0)$ \\
         \hline
         
          $(x+5)^{15}(x-3)(x-5)^8(x-7)^2(x-9)(x^2-9x+16)$   \\
          $(0, 0, 0, 20451, 8413, 0, -1018)$ \\
         \hline
         
          $(x+5)^{15}(x-3)(x-4)(x-5)^6(x-7)^2(x^2-12x+31)^2$   \\
          $(-4921960, 0, 0, -3784, 3, 0, 0)$ \\
         \hline
         
          $(x+5)^{15}(x-5)^9(x-7)(x^4-23x^3+183x^2-581x+596)$   \\
          $(-899357723, 0, 0, -690887, -98738, -14106, -2015)$ \\
         \hline
         
          $(x+5)^{15}(x-5)^9(x-7)(x^4-23x^3+183x^2-565x+484)$   \\
          $(0, 0, 0, 0, -17, 0, -242)$ \\
         \hline
         
          $(x+5)^{15}(x-5)^8(x-7)^2(x^4-21x^3+151x^2-431x+388)$   \\
          $(0, 0, 0, 6028, 3367, 978, 238)$ \\
         \hline
         
          $(x+5)^{15}(x-5)^7(x-7)^3(x^4-19x^3+123x^2-313x+256)$   \\
          $(0, 0, 0, 2863, 1795, 616, 174)$ \\
         \hline
         
          $(x+5)^{15}(x-5)^9(x^5-30x^4+344x^3-1874x^2+4823x-4672)$   \\
          $(0, 0, 0, 11531, 6111, 1747, 425)$ \\
         \hline
         
          $(x+5)^{15}(x-5)^8(x-7)(x^5-28x^4+298x^3-1500x^2+3557x-3176)$   \\
          $(-391700060, 0, 0, -184580, -25940, -3625, -500, 4)$ \\
         \hline
         
          $(x+5)^{15}(x-5)^8(x-7)(x^5-28x^4+298x^3-1484x^2+3365x-2616)$   \\
          $(3329450510, 0, 0, 1568930, 212711, 25343, 2250, 0)$ \\
         \hline
         
          $(x+5)^{15}(x-5)^8(x-7)(x^2-10x+17)(x^3-18x^2+101x-176)$   \\
          $(0, 0, 0, 0, -27929, -19411, -7066, -2001)$ \\
         \hline
         
             \caption{ Certificates of infeasibility for each polynomial in $P_{29,14} \backslash E_{29,14}$.
              (The computation to find the interlacing characteristic polynomials in each case took less than 21.80 seconds on a modern PC running SageMath~\cite{github}.)}
    \label{tab:25polydim14}
\end{longtable}


In Table~\ref{tab:20polydim16}, we list all but two of the polynomials in the set $P_{41,16}$  together with their certificates of infeasibility.

\begin{longtable}{l}

\endfirsthead

\multicolumn{1}{c}%
{{\bfseries \tablename\ \thetable{} -- continued from previous page}} \\ \hline 
\endhead

\hline \multicolumn{1}{r}{{Continued on next page}} \\ \hline
\endfoot

\endlastfoot

          $(x+5)^{25}(x-7)^{10}(x-9)^4(x^2-19x+80)$   \\
          $(-26642928,0,0,-26341,-4390)$ \\
         \hline

          $(x+5)^{25}(x-7)^{12}(x-9)^2(x-11)(x-12)$   \\
          $(32108784, 0, 0, 18832, 1713)$ \\
         \hline
         
          $(x+5)^{25}(x-7)^{12}(x-8)(x-11)^3$   \\
          $(-13493, 0, 0, -62)$ \\
         \hline
         
          $(x+5)^{25}(x-7)^9(x-9)^3(x^4-35x^3+447x^2-2465x+4948)$   \\
          $(0, 0, 0, -427982, -149845, -31263, -5561)$ \\
         \hline
         
          $(x+5)^{25}(x-7)^9(x-9)^4(x^3-26x^2+213x-544)$   \\
          $(241566984, 0, 0, 99410, 11046, 1227)$ \\
         \hline
         
          $(x+5)^{25}(x-7)^{10}(x-9)^3(x-11)(x^2-17x+64)$   \\
          $(0, 0, 0, 26648, 10679, 2150)$ \\
         \hline
         
          $(x+5)^{25}(x-7)^{10}(x-9)^2(x-11)^2(x^2-15x+52)$   \\
          $(0, 0, 0, 38390, 15614, 2995)$ \\
         \hline
         
          $(x+5)^{25}(x-7)^8(x-8)(x-9)^2(x-11)(x^2-16x+59)^2$   \\
          $(0, 0, 0, -922821, -295413, -53247, -8185)$ \\
         \hline
         
          $(x+5)^{25}(x-7)^8(x-9)^5(x^3-24x^2+179x-412)$   \\
          $(356804550, 0, 0, 188720, 26961, 3851)$ \\
         \hline
         
          $(x+5)^{25}(x-5)(x-7)^8(x-9)^4(x-11)(x^2-17x+68)$   \\
          $(0, 0, 0, 0, 99745, 50373, 10481)$ \\
         \hline
         
           $(x+5)^{25}(x-7)^7(x-9)^4(x^5-40x^4+626x^3-4784x^2+17829x-25904)$   \\
          $(0, 0, 0, 0, 0, 19125, 14174, 5890)$ \\
         \hline
         
           $(x+5)^{25}(x-7)^7(x-9)^4(x^5-40x^4+626x^3-4780x^2+17757x-25580)$  \\
          $(2198761371082, 0, 0, 316965325, 29410745, 2438199, 152389, 0)$ \\
         \hline
         
           $(x+5)^{25}(x-7)^6(x-9)^4(x^2-15x+52)(x^2-16x+59)^2$   \\
          $(-3905464984, 0, 0, -889825, -56453, 3, 0)$ \\
         \hline
         
           $(x+5)^{25}(x-7)^7(x-8)(x-9)^5(x^3-23x^2+163x-349)$   \\
          $(0, 0, 0, 0, -48443, -28265, -8404)$ \\
         \hline
         
            $(x+5)^{25}(x-5)^2(x-7)^6(x-8)(x-9)^6(x-11)$   \\
          $(1339775890, 0, 0, 659363, 82421, 10302)$ \\
         \hline
         
            $(x+5)^{25}(x-7)^7(x-9)^5(x^4-31x^3+347x^2-1653x+2824)$   \\
          $(0, 0, 0, 0, 0, -61, -67)$ \\
         \hline
         
             $(x+5)^{25}(x-4)(x-7)^8(x-9)^6(x-11)$   \\
          $(2871696, 0, 0, 3493, 0)$ \\
         \hline
         
             $(x+5)^{25}(x-5)^2(x-7)^4(x-9)^8(x^2-15x+52)$   \\
          $(26404920, 0, 0, 15610, 1163, 14)$ \\
         \hline
         
             $(x+5)^{25}(x-7)^6(x-9)^7(x^3-20x^2+123x-232)$   \\
          $(0, 0, 0, -4544, -2238, -513)$ \\
         \hline
         
             $(x+5)^{25}(x-7)^7(x-9)^6(x^3-22x^2+149x-292)$   \\
          $(0, 0, 0, -3873, -1613, -180)$ \\
          \hline 
         
         \caption{Certificates of infeasibility for each polynomial in $P_{41,16}\backslash E_{41,16}$.
              (The computation to find the interlacing characteristic polynomials in each case took less than 12.41 seconds on a modern PC running SageMath~\cite{github}.)}
    \label{tab:20polydim16}
    \end{longtable}

\begin{thebibliography}{99}
		
		\bibitem{Azarija75}
		J. Azarija, T. Marc, \emph{There is no (75,32,10,16) strongly regular graph},
		Linear Algebra Appl. \textbf{557}, 62--83 (2018).
		
		\bibitem{Azarija95}
		J. Azarija, T. Marc, \emph{There is no (95, 40, 12, 20) strongly regular graph}, J. Combin. Des., 1--13 (2019). 
		
		\bibitem{balla18} 
		{ I. Balla, F. Draxler, P. Keevash, and B. Sudakov}, \emph{Equiangular lines and spherical codes in Euclidean space}, Invent. Math., \textbf{211}(1), 179--212 (2018).
		
		
		\bibitem{bukh16}
		{ B. Bukh}, \emph{Bounds on equiangular lines and on related spherical codes}, SIAM J. Discrete Math.,
		\textbf{30}, 549--554 (2016).
		
	 \bibitem{Cau:Interlace}
	 { A.-L. Cauchy},
	 \newblock \emph{Sur l'{\'e}quation {\`a} l'aide de laquelle on d{\'e}termine les
	 	   in{\'e}galit{\'e}s s{\'e}culaires des mouvements des plan\`{e}tes}.
	 \newblock In {Oeuvres compl\`{e}tes, IIi{\`e}me S{\'e}rie}.
	   Gauthier-Villars, 1829.
		
		\bibitem{crs97} 
		{ D. Cvetkovi\'c, P. Rowlinson, and S. Simi\'c}, \emph{Eigenspaces of graphs}, Encyclopedia of Mathematics and its Applications, vol. 66, Cambridge University Press, Cambridge, (1997).	
		
		
		\bibitem{tao}
		 P.B. Denton, S.J. Parke, T. Tao, X. Zhang. Eigenvectors from Eigenvalues: a survey of a basic identity in linear algebra. arXiv:1908.03795. 2019.
		
		\bibitem{farkas}
		{ J. Farkas},
		\emph{Theorie der einfachen Ungleichungen}, Journal für die reine und angewandte Mathematik \textbf{124}, 1--27 (1902).
			
		\bibitem{FDC55}
		{ R.A. Frazer, W.J. Duncan, and A.R. Collar}, 
		 \emph{Elementary Matrices and Some Applications to Dynamics and Differential Equations}, 
		 Cambridge University Press, Cambridge, England, (1955).


		 \bibitem{Fisk:Interlace05}
		 { S. Fisk},
		 \newblock \emph{A very short proof of {C}auchy's interlace theorem for eigenvalues of
		 		   {H}ermitian matrices}.
		 \newblock {Amer. Math. Monthly}, \textbf{112}(2):118--118 (2005).
		 
 		\bibitem{GoMK}
 		{ C.D. Godsil and B.D. McKay},
 		 \emph{Spectral conditions for the reconstructibility of a graph},
 		 J. Combin. Theory Ser. B, \textbf{30}(3), 285--289 (1981).
 		 
 		 \bibitem{GlazYu}
 		 A. Glazyrin and W.-H. Yu. \emph{Upper bounds for s-distance sets and equiangular lines}.
Adv. Math. \textbf{330}:810–833 (2018).

        \bibitem{github} 
		 { G.R.W. Greaves}, SageMath notebook containing computations from the paper ``Equiangular lines in low dimensional Euclidean spaces''. \url{https://github.com/grwgrvs/interlacing-characteristic-polynomials/}.
		 
		 \bibitem{GG18} 
		 { G.R.W. Greaves}, \emph{Equiangular line systems and switching classes containing regular graphs}, Linear Algebra Appl. \textbf{536}, 31--51 (2018).
		 
		  \bibitem{GKMS16} 
		  { G. Greaves, J. Koolen, A. Munemasa and F. Sz\" oll\H osi}, \emph{Equiangular lines in Euclidean spaces}, J. Combin. Theory Ser. A \textbf{138}, 208--235 (2016).
		  
		  \bibitem{GreavesYatsyna19} 
		 { G.R.W. Greaves and P. Yatsyna}, \emph{On equiangular lines in 17 dimensions and the characteristic polynomial of a Seidel matrix}. Math. Comp. \textbf{88}, 3041--3061 (2019).
		  
		  \bibitem{haantjes48}
		  { J. Haantjes},\emph{ Equilateral point-sets in elliptic two- and three-dimensional spaces}, Nieuw Arch. Wiskd. \textbf{22}, 355--362 (1948).
    
        \bibitem{heath}
        R.W. Heath and T. Strohmer, \emph{Grassmannian frames with applications to coding
and communication}. Appl. Comput. Harmon. Anal., \textbf{14}, 257--275 (2003).
		 
		 \bibitem{Hwang:Interlace04}
		 { S.-G. Hwang},
		 \newblock \emph{Cauchy's interlace theorem for eigenvalues of {H}ermitian matrices}.
		 \newblock {Amer. Math. Monthly}, 111:157--159 (2004).
		 
		 \bibitem{jiang}
		 Z. Jiang, J. Tidor, Y. Yao, S. Zhang, and Y. Zhao. Equiangular lines with a fixed angle. arXiv:1907.12466, 2019.
		 
		 \bibitem{KingTang} { E. King and X. Tang}, \emph{New upper bounds for equiangular lines by pillar decomposition}, SIAM J. Discrete Math. 
		 \textbf{33}(4), 2479--2508 (2019).

        \bibitem{delaat}
        {D. de Laat, F.C. Machado, F. M. de Oliveira Filho and F. Vallentin.} $k$–point semidefinite
programming bound for equiangular lines, arXiv:1812.06045, 2018.

		 
		 \bibitem{lemmens73}
		 { P.W.H. Lemmens and J.J. Seidel},
		 \emph{Equiangular lines},
		 { J. Algebra},
		 {\bf 24}, 494--512 (1973).
     
     \bibitem{LinYu} { Y.-C. Lin and W.-H. Yu}, \emph{Saturated configuration and new large construction of equiangular lines}, Linear Algebra Appl. \textbf{588}, 272--281 (2020).
		 
		 \bibitem{vLintSeidel66}
		 { J.H. van Lint and J.J. Seidel}, \emph{Equilateral point sets in elliptic geometry}, Indag. Math., \textbf{28}, 335--348 (1966).
		 
		 \bibitem{mathematica}
		 Wolfram Research, Inc., Mathematica, Version Number 10.0, Champaign, IL (2014).

		 
		 \bibitem{MS04}
		 { J.F. McKee and C.J. Smyth},
		 \emph{Salem numbers of trace $-2$ and traces of totally positive algebraic integers}. Algorithmic number theory, 327--337, Lecture Notes in Comput. Sci., 3076, Springer, Berlin, 2004.
		 
		 \bibitem{okudayu16} 
		 { T. Okuda and W.-H. Yu}, \emph{A new relative bound for equiangular lines and nonexistence of tight spherical designs of harmonic index 4}, European J. Combin., \textbf{53}, 96--103 (2016).
		 		 
		 \bibitem{parigp}
    The PARI~Group, PARI/GP version {\tt 2.10}, Univ. Bordeaux,
    \url{http://pari.math.u-bordeaux.fr/}.
    
             \bibitem{thompson1968principal}
         { R.C. Thompson}, \emph{Principal submatrices V: Some results concerning principal submatrices of arbitrary matrices}, J. Res. Nat. Bur. Standards Sect. B, \textbf{72}(2), 115--125 (1968).
		 
		 \bibitem{Seidel74}
		 { J.J. Seidel}, \emph{Graphs and two-graphs}, in: Proc. 5th Southeastern Conf. on Combinatorics, Graph Theory, and Computing, Utilitas Mathematica Publishing Inc., Winnipeg, Canada, 1974.
        
        \bibitem{SeidelIncorrectTable}
            J. J. Seidel: ‘Discrete non-Euclidean geometry’, In Buekenhout (ed.), Handbook of Incidence
Geometry, Elsevier Science, Amsterdam, The Netherlands (1995).
		 
		 \bibitem{oeis} N.J.A. Sloane, editor. The On-Line Encyclopedia of Integer Sequences. Published electronically at \url{https://oeis.org/}.
		 
 		 \bibitem{Smyth84}
 		 { C.J. Smyth}, 
 		 \emph{Totally positive algebraic integers of small trace}, 
 		 Ann. Inst. Fourier, Grenoble, \textbf{33}(3), 1--28 (1984).
		 
		 \bibitem{sage} { W.A. Stein et al}. \emph{Sage Mathematics Software (Version 8.8)}. The Sage
		 Development Team, 2010. \url{http://www.sagemath.org/}.
		 
		 \bibitem{Szo} F. Sz\" oll\H osi, \emph{A remark on a construction of D. S. Asche}, Discrete Comput. Geom. \textbf{61}(1), 120--122 (2019).
		 
		 \bibitem{taylor}
		 D.E. Taylor, \emph{Regular 2-graphs}, Proc. London Math. Soc. \textbf{35}, 257–274 (1977).
		 
		 \bibitem{Yu17} { W.-H. Yu}, \emph{New bounds for equiangular lines and spherical two-distance sets}, SIAM J. Discrete Math. 
		 \textbf{31}(2), 908--917 (2017).
		
		\end{thebibliography}
\end{document}